\documentclass[final]{siamltex}

\newtheorem{remark}{{\it Remark}}[section]

\title{Global classical solutions to the compressible Euler-Maxwell equations}

\author{Jiang Xu\thanks{Department of Mathematics, Nanjing
University of Aeronautics and Astronautics, Nanjing 211106,
P.R.China ({\tt jiangxu\underline{ }79@yahoo.com.cn}, {\tt
jiangxu\underline{ }79@nuaa.edu.cn}).} }

\begin{document}

\maketitle

\begin{abstract}
In this paper, we consider the compressible Euler-Maxwell equations
arising in semiconductor physics, which take the form of Euler
equations for the conservation laws of mass density and current
density for electrons, coupled to Maxwell's equations for
self-consistent electromagnetic field. We study the global
well-posedness in critical spaces and the limit to zero of some
physical parameters in the scaled Euler-Maxwell equations. More
precisely, using high- and low-frequency decomposition methods, we
first construct uniform (global) classical solutions (around
constant equilibrium) to the Cauchy problem of Euler-Maxwell
equations in Chemin-Lerner's spaces with critical regularity.
Furthermore, based on Aubin-Lions compactness lemma, it is justified
that the (scaled) classical solutions converge globally in time to
the solutions of compressible Euler-Poisson equations in the process
of non-relativistic limit and to that of drift-diffusion equations
under the relaxation limit or the combined non-relativistic and
relaxation limits.
\end{abstract}

\begin{keywords}
Euler-Maxwell equations, classical solutions, Chemin-Lerner's
spaces, non-relativistic limit, relaxation limit
\end{keywords}

\begin{AMS}
35L45, 76N15, 35B25
\end{AMS}

\pagestyle{myheadings} \thispagestyle{plain} \markboth{Jiang
Xu}{Compressible Euler-Maxwell equations}

\section{Introduction and main results}
The increasing demand on semiconductor devices has led to the
necessity of a deep and detailed understanding on the mathematical
theory of various charge-carrier transport models. Of these
important models, the classical hydrodynamic model (also named as
the Euler-Poisson equations), which treats the propagation of
electrons in semiconductor devices as the flow of a compressible
charged fluid in an electric field, has received increasing
attention. For the cases of high electric field and submicronic
devices, the Euler-Poisson equations of fluid dynamical form can
represent a reasonable comprise between physical accuracy and
reduction of computational cost in real applications, the reader is
referred to \cite{MRS} for more explanations. When semiconductor
devices are operated under some high frequency conditions (such as
photoconductive switches, electro-optics, semiconductor lasers and
high-speed computers), magnetic fields are generated by moving
electrons inside devices, then the electrons transport interacts
with the propagating electromagnetic waves. In this case, the
transport process is typically governed by the Euler-Maxwell
equations, which is more accurate than the Euler-Poisson equations,
since the electromagnetic field obeys Maxwell's equations instead of
Poisson equation for the electric field only.

After some appropriate re-scaling, the compressible Euler-Maxwell
equations are written, in nondimensional form,  as
\begin{equation}
\left\{
\begin{array}{l}\partial_{t}n+\nabla\cdot(n\mathbf{u})=0,\\
 \partial_{t}(n\textbf{u})+\nabla\cdot(n\textbf{u}\otimes\textbf{u})+\nabla P(n)=-n(\mathbf{E}+\varepsilon\textbf{u}\times
 \mathbf{B})-\frac{n\textbf{u}}{\tau},\\
\varepsilon\lambda^2\partial_{t}\mathbf{E}-\nabla\times\mathbf{B}=\varepsilon
n\textbf{u},\\
\varepsilon\partial_{t}\mathbf{B}+\nabla\times\mathbf{E}=0,\\
\lambda^2\nabla\cdot\mathbf{E}=\bar{n}-n,\ \ \ \nabla\cdot\mathbf{B}=0,\\
 \end{array} \right.\label{R-E1}
\end{equation}
 for $ (t,x)\in[0,+\infty)\times\mathbf{R}^{N}(N\geq2)$.
Here the unknowns
$n,\textbf{u}=(u_1,u_2,\cdot\cdot\cdot,u_{N})^{\top},\
\mathbf{E}=(E_1,E_2,\cdot\cdot\cdot,E_{N})^{\top},\
\mathbf{B}=(B_1,B_2,\cdot\cdot\cdot,B_{N})^{\top}(\top$ transpose)
denote the electron density, electron
 velocity, electric field and magnetic field, respectively.
The pressure $P(n)$ satisfies the usual $\gamma$-law
\begin{eqnarray}
P(n)=P_{0}n^{\gamma}(\gamma\geq1),\label{R-E2}
\end{eqnarray}
where $P_{0}>0$ is some physical constant.  The system (\ref{R-E1})
is called \emph{isentropic} if $\gamma>1$ and \emph{isothermal} if
$\gamma=1$. $\tau,\lambda>0$ are the (scaled) constants for the
momentum-relaxation time and the Debye length.
$c=(\epsilon_{0}\upsilon_{0})^{-\frac{1}{2}}>0$ is the speed of
light, where $\epsilon_{0}$ and $\upsilon_{0}$ are the vacuum
permittivity and permeability. Setting $\varepsilon=\frac{1}{c}$.
The independent parameters $\tau,\lambda$ and $\varepsilon$ which
arise from nondimensionalization, are assumed to be very small
compared to the reference physical size. The symbols $\nabla$,
$\cdot$, $\times$ and $\otimes$ are the gradient operator, the
scalar products, the vector products and the tensor products of two
vectors, respectively. $\bar{n}>0$ is the doping profile, which
stands for the density of positively charged background ions.

It is not difficult to see that the above Euler-Maxwell equations
consist of a quasi-linear hyperbolic system, the main feature of
which is the finite time blowup of classical solutions even when the
initial data are smooth and small. Hence, the qualitative study and
device simulation of (\ref{R-E1}) are far to be trivial. In this
paper, our main aim is to establish the global well-posedness and
justify some singular limits for the Cauchy problem. For this
purpose, the Euler-Maxwell equations (\ref{R-E1}) are equipped with
the following initial conditions for $n,\textbf{u},\textbf{E}$ and
$\textbf{B}$:
\begin{equation}(n,\textbf{u},\textbf{E},\textbf{B})(x,0)=(n_{0},\textbf{u}_{0},\textbf{E}_{0},\textbf{B}_{0})(x),\
\ x\in\mathbf{R}^{N},\label{R-E3} \end{equation} which satisfies the
compatible conditions
\begin{equation}
\lambda^2\nabla\cdot\mathbf{E}_{0}=\bar{n}-n_{0},\ \ \
\nabla\cdot\mathbf{B}_{0}=0, \ \ x\in\mathbf{R}^{N}.\label{R-E4}
\end{equation}

\subsection{Singular limit analysis} \label{sec:1.1}
It is convenient to state previous works and main results of this
paper, we first introduce some singular limits in the scaled
Euler-Maxwell equations at the formal level, including the
non-relativistic limit, relaxation limit as well as combined
non-relativistic and relaxation limits.

Firstly, we observe the non-relativistic limit $(\textit{i.e.}\
\varepsilon\rightarrow0)$. Let $\tau=1=\lambda$ and
$(n^{\varepsilon}, \textbf{u}^{\varepsilon},
\textbf{E}^{\varepsilon}, \textbf{B}^{\varepsilon})$ be the solution
of the following equations
\begin{equation}
\left\{
\begin{array}{l}\partial_{t}n^{\varepsilon}+\nabla\cdot(n^{\varepsilon}\mathbf{u}^{\varepsilon})=0,\\[0.5mm]
 \partial_{t}(n^{\varepsilon}\textbf{u}^{\varepsilon})+\nabla\cdot(n^{\varepsilon}\textbf{u}^{\varepsilon}\otimes\textbf{u}^{\varepsilon})+\nabla P(n^{\varepsilon})\\ \hspace{30mm}=-n^{\varepsilon}(\mathbf{E}^{\varepsilon}+\varepsilon\textbf{u}^{\varepsilon}\times
 \mathbf{B}^{\varepsilon})-n^{\varepsilon}\textbf{u}^{\varepsilon},\\[0.5mm]
\varepsilon\partial_{t}\mathbf{E}^{\varepsilon}-\nabla\times\mathbf{B}^{\varepsilon}=\varepsilon
n^{\varepsilon}\textbf{u}^{\varepsilon},\\[0.5mm]
\varepsilon\partial_{t}\mathbf{B}^{\varepsilon}+\nabla\times\mathbf{E}^{\varepsilon}=0,\\[0.5mm]
\nabla\cdot\mathbf{E}^{\varepsilon}=\bar{n}-n^{\varepsilon},\ \ \
\nabla\cdot\mathbf{B}^{\varepsilon}=0.
 \end{array} \right.\label{R-E5}
\end{equation}
Formally, we see that the limits
$n^{0},\mathbf{u}^{0},\mathbf{E}^{0}$ of
$n^{\varepsilon},\mathbf{u}^{\varepsilon}, \mathbf{E}^{\varepsilon}$
as $\varepsilon\rightarrow0$ satisfy
\begin{equation}
\left\{
\begin{array}{l}\partial_{t}n^{0}+\nabla\cdot(n^{0}\mathbf{u}^{0})=0,\\
 \partial_{t}(n^{0}\textbf{u}^{0})+\nabla\cdot(n^{0}\textbf{u}^{0}\otimes\textbf{u}^{0})+\nabla P(n^{0})=-n^{0}\mathbf{E}^{0}-n^{0}\textbf{u}^{0},\\
\nabla\cdot\mathbf{E}^{0}=\bar{n}-n^{0},\ \ \ \nabla\times\mathbf{E}^{0}=0,\\
\end{array} \right.\label{R-E6}
\end{equation}
which is the well-known Euler-Poisson equations for semiconductors.
The irrotationality of $\mathbf{E}^{0}$ implies the existence of a
potential function $\mathit\Phi^{0}$ such that
$\mathbf{E}^{0}=-\nabla\mathit\Phi^{0}$. Then using the Green's
formulation, (\ref{R-E6}) can be reduced to the form of the
conservation law with a non-local source term, e.g., see \cite{HMW}.

Secondly, we justify the relaxation limit $(\textit{i.e.}\
\tau\rightarrow0)$ in the Euler-Maxwell equations (\ref{R-E1}). The
diffusion limit was first introduced by Marcati and Natalini
\cite{MN} for the Euler-Poisson equations (\ref{R-E6}). Set
$\varepsilon=1=\lambda$. To do this, as in \cite{MN}, we define the
following scaled transform
\begin{equation}(n^{\tau},\textbf{u}^{\tau},\mathbf{E}^{\tau},\mathbf{B}^{\tau})(t,x)
=\Big(n,\frac{1}{\tau}\textbf{u},\mathbf{E},\mathbf{B}\Big)\Big(\frac{t}{\tau},x\Big).\label{R-E7}
\end{equation}
Then the new variable $(n^{\tau}, \textbf{u}^{\tau},
\textbf{E}^{\tau}, \textbf{B}^{\tau})$ satisfies
\begin{equation}
\left\{
\begin{array}{l}\partial_{t}n^{\tau}+\nabla\cdot(n^{\tau}\mathbf{u}^{\tau})=0,\\
 \tau^2\partial_{t}(n^{\tau}\textbf{u}^{\tau})+\tau^2\nabla\cdot(n^{\tau}\textbf{u}^{\tau}\otimes\textbf{u}^{\tau})+\nabla P(n^{\tau})\\
 \hspace{35mm}=-n^{\tau}(\mathbf{E}^{\tau}+\tau\textbf{u}^{\tau}\times
 \mathbf{B}^{\tau})-n^{\tau}\textbf{u}^{\tau},\\
\tau\partial_{t}\mathbf{E}^{\tau}-\nabla\times\mathbf{B}^{\tau}=
\tau n^{\tau}\textbf{u}^{\tau},\\
\tau\partial_{t}\mathbf{B}^{\tau}+\nabla\times\mathbf{E}^{\tau}=0,\\
\nabla\cdot\mathbf{E}^{\tau}=\bar{n}-n^{\tau},\ \ \
\nabla\cdot\mathbf{B}^{\tau}=0.
 \end{array} \right.\label{R-E8}
\end{equation}
Formally, the limits $\mathcal{N},\mathcal{E}$ of
$n^{\tau},\textbf{E}^{\tau}$ as $\tau\rightarrow0$ satisfy the
so-called drift-diffusion equations
\begin{equation}
\left\{
\begin{array}{l}\partial_{t}\mathcal{N}=\nabla\cdot(\nabla P(\mathcal{N})+\mathcal{N}\mathcal{E}),\\
\nabla\cdot\mathcal{E}=\bar{n}-\mathcal{N},\ \
\nabla\times\mathcal{E}=0,\\
\mathcal{N}(0,x)=n_{0},
 \end{array} \right.\label{R-E9}
\end{equation}
which is a system of diffusion equations for the electron density,
and maintains the parabolic-elliptic character.

Lastly, we study the combined non-relativistic and relaxation limits
in the Euler-Maxwell equations (\ref{R-E1}) $(\textit{i.e.}\
\varepsilon,\tau \rightarrow0)$. Set $\lambda=1$. From the
``$\mathcal{O}(1/\tau)$ time scale" in (\ref{R-E7}), where the
superscript $\tau$ is replaced by $(\tau,\varepsilon)$, the new
variable $(n^{(\tau,\varepsilon)}, \textbf{u}^{(\tau,\varepsilon)},
\textbf{E}^{(\tau,\varepsilon)},\\ \textbf{B}^{(\tau,\varepsilon)})$
satisfies
\begin{equation}
\left\{
\begin{array}{l}\partial_{t}n^{(\tau,\varepsilon)}+\nabla\cdot(n^{(\tau,\varepsilon)}\mathbf{u}^{(\tau,\varepsilon)})=0,\\[1mm]
 \tau^2\partial_{t}(n^{(\tau,\varepsilon)}\textbf{u}^{(\tau,\varepsilon)})+\tau^2\nabla\cdot(n^{(\tau,\varepsilon)}\textbf{u}^{(\tau,\varepsilon)}\otimes\textbf{u}^{(\tau,\varepsilon)})+\nabla P(n^{(\tau,\varepsilon)})
 \nonumber\\ \hspace{25mm}=-n^{(\tau,\varepsilon)}\Big(\mathbf{E}^{(\tau,\varepsilon)}+\tau\varepsilon\textbf{u}^{(\tau,\varepsilon)}\times
 \mathbf{B}^{(\tau,\varepsilon)}\Big)-n^{(\tau,\varepsilon)}\textbf{u}^{(\tau,\varepsilon)},\\[1mm]
\tau\varepsilon\partial_{t}\mathbf{E}^{(\tau,\varepsilon)}-\nabla\times\mathbf{B}^{(\tau,\varepsilon)}=
\tau\varepsilon n^{(\tau,\varepsilon)}\textbf{u}^{(\tau,\varepsilon)},\\[0.5mm]
\tau\varepsilon\partial_{t}\mathbf{B}^{(\tau,\varepsilon)}+\nabla\times\mathbf{E}^{(\tau,\varepsilon)}=0,\\[1mm]
\nabla\cdot\mathbf{E}^{(\tau,\varepsilon)}=\bar{n}-n^{(\tau,\varepsilon)},\
\ \ \nabla\cdot\mathbf{B}^{(\tau,\varepsilon)}=0.
 \end{array} \right.\label{R-E10}
\end{equation}
Obviously, in the process of combined limits
$\tau,\varepsilon\rightarrow0$, the limits $\mathcal{N},\mathcal{E}$
of $n^{(\tau,\varepsilon)},\textbf{E}^{(\tau,\varepsilon)}$ also
satisfy the drift-diffusion equations (\ref{R-E9}).

\subsection{Main results}\label{sec:1.2}
In the past ten years, the Euler-Poisson equations (\ref{R-E6}) have
attracted much attention. There are many contributions in
mathematical analysis, such as the well-posedness of steady-state
solutions, global existence of classical or entropy weak solutions,
large time behavior of classical solutions, relaxation limit
problems and so on, the reader is referred to
\cite{A,DM,G,Gu,HMW,HJZ,HZ,LMM,MN} and the references therein, also
including ourselves \cite{FXZ,X,XY}, while the Euler-Maxwell
equations are much more intricate than the Euler-Poisson equations,
not only because of Maxwell's equations, but also because of the
complicated coupling of the Lorentz force
$(\textbf{E}+\textbf{u}\times \textbf{B})$. In contrast, not so many
works have been devoted to the study of Euler-Maxwell equations. Up
to now, only partial results are available.

Using the Godunov scheme with the fractional step and the
compensated compactness theory, Chen, Jerome and Wang \cite{CJW}
constructed the existence of a global weak solution to the initial
boundary value problem for arbitrarily large initial data in
$L^\infty(\mathbf{R})$. In \cite{J}, assuming initial data in
Sobolev spaces $H^s(\mathbf{R}^3)$ with higher regularity ($s>5/2$),
a local existence theory of smooth solutions for the Cauchy problem
of non-isentropic Euler-Maxwell equations, where the
pressure-density function $(\ref{R-E2})$ is replaced with the energy
equation, was established by modificating the classical
semigroup-resolvent approach of Kato \cite{K}. In
\cite{PW1,PW2,PW3}, based on the existence theory of Kato and Majda
\cite{K,M}, Peng and Wang justified the non-relativistic limit
$(\varepsilon\rightarrow0)$, the quasi-neutral limit
$(\lambda\rightarrow0)$ and the combined non-relativistic and
quasi-neutral limits $(\varepsilon=\lambda\rightarrow0)$ for the
Euler-Maxwell equations (\ref{R-E1}) in virtue of the analysis of
asymptotic expansions. Their results show that the Euler-Maxwell
equations converge towards the Euler-Poisson equations, e-MHD system
and incompressible Euler equations in some time-interval independent
of the parameters $\varepsilon$ and $\lambda$, respectively.

However, the well-posedness and singular limits for the
Euler-Maxwell equations (\ref{R-E1}) in several dimensions are still
far from well-known, in particular, in the framework of critical
spaces. In the present paper, we shall answer this problem. More
concretely speaking, we shall consider a small perturbation near the
constant equilibrium state
$(\bar{n},\mathbf{0},\mathbf{0},\overline{\mathbf{B}})$ which is a
particular solution of the Cauchy problem (\ref{R-E1})-(\ref{R-E3}),
and obtain the global existence and uniqueness of classical
solutions. We choose the critical Besov spaces in space-variable $x$
as the basic functional setting, where the regularity index
$(\sigma=1+N/2)$ is just the limit case of classical existence
theory of Kato and Majda \cite{K,M}. Although this idea has been
used to study the compressible Euler-Poisson equations (\ref{R-E6})
in \cite{FXZ,X,XY} recently, it should be pointed out that the
Euler-Maxwell equations are essentially different from (\ref{R-E6}).
In comparison with the methods in \cite{FXZ,X,XY}, we have to face
with several technical difficulties arising in the uniform \textit{a
priori} estimates of classical solutions in critical spaces. The
first one is lack of the low-frequency estimate of magnetic field
$\mathbf{B}$, which does not lead to the exponential decay near
equilibrium in view of the standard definition of norm of Besov
spaces. Another one is that the nonlinear terms (pressure, Lorentz
field, etc.) will hinder us establishing the uniform estimates with
respect to the singular parameter couple $(\tau,\varepsilon)$. To
overcome these difficulties, we add the new content in the proof of
the local existence and (uniform) global existence of classical
solutions. Actually, the Chemin-Lerner's spaces
$\widetilde{L}^{\rho}_{T}(B^{s}_{p,r})$ in \cite{C2} are introduced,
which is a refinement of the usual spaces
$L^{\rho}_{T}(B^{s}_{p,r})$, and some uniform frequency-localization
estimates in Chemin-Lerner's spaces with critical regularity are
established, for details, see Lemmas \ref{lem4.1}-\ref{lem4.2} and
Lemmas \ref{lem4.4}-\ref{lem4.5}. Based on the uniform estimates, we
further rigorously justify the singular limit problems for
(\ref{R-E1})-(\ref{R-E3}) in Sect.~\ref{sec:1.1} by the standard
weak convergence methods and the application of compactness theorem
in \cite{S}.

Throughout this paper, the regularity index $\sigma=1+N/2$. First of
all, we state a local existence and uniqueness theorem of classical
solutions to the Cauchy problem (\ref{R-E1})-(\ref{R-E3}) away from
the vacuum.

\begin{theorem}\label{thm1.1} Let $\bar{n}>0$ be a constant reference density and
$\overline{\mathbf{B}}\in \mathbf{R}^{N}$ be any given constant.
Suppose that $n_{0}-\bar{n}, \mathbf{u}_{0},\mathbf{E}_{0}$ and
$\mathbf{B}_{0}-\overline{\mathbf{B}}\in
B^{\sigma}_{2,1}(\mathbf{R}^{N})$ satisfy $n_{0}>0$ and the
compatible conditions (\ref{R-E4}), then there exist a time
$T_{0}>0$ and a unique solution
$(n,\mathbf{u},\mathbf{E},\mathbf{B})$ of the system
(\ref{R-E1})-(\ref{R-E3}) such that
$$(n,\mathbf{u},\mathbf{E},\mathbf{B})\in \mathcal{C}^{1}([0,T_{0}]\times \mathbf{R}^{N})\ \ \ \mbox{with}\
\ \  n>0\ \  \mbox{for all} \ \ t\in [0,T_{0}]$$ and
$$(n-\bar{n},\mathbf{u},\mathbf{E}, \mathbf{B}-\overline{\mathbf{B}})\in
\widetilde{\mathcal{C}}_{T_{0}}(B^{\sigma}_{2,1})\cap
\widetilde{\mathcal{C}}^1_{T_{0}}(B^{\sigma-1}_{2,1}).$$
\end{theorem}

\begin{remark}
\rm To avoid excessive commutators arising from the nonlinear
pressure term by using Fourier frequency-localization method, we
introduce a function transform in Sect.~\ref{sec:3.1} such that the
Euler-Maxwell equations (\ref{R-E1}) is reduced to a symmetric
hyperbolic system. Based on the previous effort in \cite{FXZ}, we
obtain the local existence of classical solutions in the
Chemin-Lerner's space with critical regularity (Proposition
\ref{prop3.1}). Theorem \ref{thm1.1} follows from Proposition
\ref{prop3.1} and Remark \ref{rem3.1} readily. As a matter of fact,
the new result is applicable to generally symmetrizable hyperbolic
systems, which enriches and develops the classical existence theory
of Kato and Majda \cite{K,M}.
\end{remark}

In small amplitude regime, we get the uniform global well-posedness
of classical solutions to the Cauchy problem
(\ref{R-E1})-(\ref{R-E3}) in critical spaces. From now on, we set
the scaled Debye length to be one ($\lambda\equiv1$).

\begin{theorem}\label{thm1.2}
 Let $\bar{n}>0$ be a constant reference density and
$\overline{\mathbf{B}}\in \mathbf{R}^{N}$ be any given constant.
Suppose that $n_{0}-\bar{n}, \mathbf{u}_{0},\mathbf{E}_{0}$ and
$\mathbf{B}_{0}-\overline{\mathbf{B}}\in
B^{\sigma}_{2,1}(\mathbf{R}^{N})$ satisfy the compatible conditions
(\ref{R-E4}). There exists a positive constant $\delta_{0}$
independent of singular parameter couple $(\tau,\varepsilon)$, such
that if
\begin{eqnarray*}
\|(n_{0}-\bar{n}, \mathbf{u}_{0},\mathbf{E}_{0},
\mathbf{B}_{0}-\overline{\mathbf{B}})\|_{B^{\sigma}_{2,1}}\leq\delta_{0},
\end{eqnarray*}
then there exists a unique global solution
$(n,\mathbf{u},\mathbf{E},\mathbf{B})$ of the system
(\ref{R-E1})-(\ref{R-E3}) satisfying
\begin{eqnarray*}
(n,\mathbf{u},\mathbf{E},\mathbf{B})\in
\mathcal{C}^{1}([0,\infty)\times \mathbf{R}^{N})
\end{eqnarray*}
and
\begin{eqnarray*}
(n-\bar{n},\mathbf{u},\mathbf{E},\mathbf{B}-\overline{\mathbf{B}})
\in \widetilde{\mathcal{C}}(B^{\sigma}_{2,1}(\mathbf{R}^{N}))\cap
\widetilde{\mathcal{C}}^1(B^{\sigma-1}_{2,1}(\mathbf{R}^{N})).
\end{eqnarray*}
Moreover, the uniform energy estimate holds:
\begin{eqnarray}
&&\|(n-\bar{n},\mathbf{u}, \mathbf{E},
\mathbf{B}-\overline{\mathbf{B}})\|_{\widetilde{L}^\infty(B^{\sigma}_{2,1})}
\nonumber\\&&+\mu_{0}\Big\{\Big\|\Big(\sqrt{\tau}(n-\bar{n}),\frac{1}{\sqrt{\tau}}\mathbf{u},\sqrt{\tau\varepsilon}\mathbf{E}\Big)\Big\|_{\widetilde{L}^2(B^{\sigma}_{2,1})}
+\Big\|\frac{1}{\sqrt{\varepsilon}}\nabla\mathbf{B}\Big\|_{\widetilde{L}^2(B^{\sigma-1}_{2,1})}\Big\}
\nonumber\\&\leq& C_{0}\|(n_{0}-\bar{n},
\mathbf{u}_{0},\mathbf{E}_{0},
\mathbf{B}_{0}-\overline{\mathbf{B}})\|_{B^{\sigma}_{2,1}}\label{R-E11}
\end{eqnarray}
for $0<\tau,\varepsilon\leq1$, where the positive constants
$\mu_{0},C_{0}$ are independent of $(\tau,\varepsilon)$.
\end{theorem}

\begin{remark}
\rm  Together with Theorem \ref{thm1.1}, Theorem \ref{thm1.2}
directly follows from the standard continuation argument and the
crucial energy estimate (\ref{R-E11}) which presents the dissipation
rates of all the components in the solution. Noticing that the
coupled electromagnetic field $(\mathbf{E},\mathbf{B})$ appears in
the nonlinear source terms of Euler system, which indeed does not
affect the character of corresponding linearized form, so we can
take the full advantage of ``Shizuta-Kawashima" skew-symmetry
condition which was well developed for general hyperbolic systems of
balance laws \cite{KY,Y} to capture the dissipation rate of density
function, see Lemma \ref{lem4.2}. On the other hand, from the proof
of Lemmas \ref{lem4.4}-\ref{lem4.5}, we see that the electromagnetic
field generated by the compressible electron flow exhibits a weak
dissipation property, which is essentially different from the pure
Maxwell's equations, although the low-frequency estimate of
dissipation rate of $\mathbf{B}$ is absent. In addition, we track
the singular parameters $\tau$ and $\varepsilon$ in the proof of
(\ref{R-E11}), which plays a key role in the study of related limit
problems.
\end{remark}

As a direct consequence of Theorem \ref{thm1.2}, we can obtain the
large-time asymptotic behavior of global solutions near the
equilibrium $(\bar{n},\mathbf{0},\mathbf{0},\overline{\mathbf{B}})$
in some Besov spaces.
\begin{corollary}\label{cor1.3}
Let $(n,\mathbf{u},\mathbf{E},\mathbf{B})$ be the solution in
Theorem \ref{thm1.2}, it holds that
$$\|n(\cdot,t)-\bar{n},\mathbf{u}(\cdot,t), \mathbf{E}(\cdot,t)\|_{B^{\sigma-\varepsilon'}_{2,1}(\mathbf{R}^{N})}\rightarrow 0,
\ \ \ \
\|\mathbf{B}(\cdot,t)-\overline{\mathbf{B}}\|_{B^{\sigma-1-\varepsilon'}_{p,1}(\mathbf{R}^{N})}\rightarrow
0, $$ as the time variable $t\rightarrow +\infty$, where
$p=\frac{2N}{N-2}(N>2)$ and $\varepsilon'>0$.
\end{corollary}

\begin{remark}\label{rem1.3}
\rm Recalling the proof of Corollary 5.1 in \cite{FX}, Corollary
\ref{cor1.3} is followed by a minor revision. The definite
convergence rate to the equilibrium
$(\bar{n},\mathbf{0},\mathbf{0},\overline{\mathbf{B}})$ will be
studied in the future work.
\end{remark}

Next, we state the non-relativistic limit of uniform global
solutions to (\ref{R-E1})-(\ref{R-E3}) for any fixed momentum
relaxation time $\tau>0$.

\begin{theorem}[\textbf{Non-relativistic
limit}]\label{thm1.3} Let $\tau=1$ and $(n^{\varepsilon},
\mathbf{u}^{\varepsilon}, \mathbf{E}^{\varepsilon},
\mathbf{B}^{\varepsilon})$ be the global solution of
(\ref{R-E1})-(\ref{R-E3}) given by Theorem \ref{thm1.2}. Then there
exists some function $(n^{0},\mathbf{u}^{0},\mathbf{E}^{0})$ which
is a global solution to the Euler-Poisson equations (\ref{R-E6})
satisfying $(n^{0}-\bar{n},\mathbf{u}^{0},\mathbf{E}^{0})\in
\mathcal{C}([0,\infty), B^{\sigma}_{2,1}(\mathbf{R}^{N}))$ such that
as $\varepsilon\rightarrow0$, it holds that
\begin{eqnarray*}
(n^{\varepsilon},\mathbf{u}^{\varepsilon},\sqrt{\varepsilon}\mathbf{E}^{\varepsilon})\rightarrow
(n^{0},\mathbf{u}^{0},\mathbf{0})\ \ \ \ \mbox{strongly in}\ \ \
\mathcal{C}([0,T],(B^{\sigma-\delta}_{2,1}(\mathbf{R}^{N}))_{\mathrm{loc}}),
\end{eqnarray*}
\begin{eqnarray*}
\nabla\mathbf{B}^{\varepsilon}\rightarrow \mathbf{0}\ \ \ \
\mbox{strongly
 in}\ \ \ L^2_{T}(B^{\sigma-1}_{2,1}(\mathbf{R}^{N})),
\end{eqnarray*}
\begin{eqnarray*}
(\mathbf{E}^{\varepsilon},\mathbf{B}^{\varepsilon})\rightharpoonup
(\mathbf{E}^{0}, \overline{\mathbf{B}})\ \ \ \
\mbox{weakly$^{\star}$
 in}\ \ \ L^\infty_{T}(B^{\sigma}_{2,1}(\mathbf{R}^{N})),
\end{eqnarray*}
for any $T>0$ and $\delta\in(0,1)$. Moreover, it yields
\begin{eqnarray}
&&\|(n^{0}-\bar{n},\mathbf{u}^{0},
\mathbf{E}^{0})(t,\cdot)\|_{B^{\sigma}_{2,1}(\mathbf{R}^{N})}\nonumber\\
&\leq& C_{1}\|(n_{0}-\bar{n}, \mathbf{u}_{0},\mathbf{E}_{0},
\mathbf{B}_{0}-\overline{\mathbf{B}})\|_{B^{\sigma}_{2,1}(\mathbf{R}^{N})},\
t\geq0, \label{R-E12}
\end{eqnarray}
where $C_{1}>0$ is a uniform constant independent of $\varepsilon$.
\end{theorem}

Secondly, we justify the relaxation limit for the Euler-Maxwell
equations (\ref{R-E1}). To this end, we consider the Cauchy problem
for the re-scaled system (\ref{R-E8}) subject to the initial data
\begin{equation}(n^{\tau},\textbf{u}^{\tau},\mathbf{E}^{\tau},\mathbf{B}^{\tau})(0,x)
=\Big(n_{0},\frac{1}{\tau}\textbf{u}_{0},\mathbf{E}_{0},\mathbf{B}_{0}\Big)(x).\label{R-E13}
\end{equation}
It follows from Theorem \ref{thm1.2} and the ``$\mathcal{O}(1/\tau)$
time scale" (\ref{R-E7}) that there exists a unique global in-time
classical solution $(n^{\tau},\textbf{u}^{\tau},\mathbf{E}^{\tau},
\mathbf{B}^{\tau})$ to the system (\ref{R-E8}) and (\ref{R-E13}).
Then, we have
\begin{theorem}[\textbf{Relaxation limit}]\label{thm1.4}
Let $\varepsilon=1$ and $(n^{\tau}, \mathbf{u}^{\tau},
\mathbf{E}^{\tau}, \mathbf{B}^{\tau})$ be the global solution of
(\ref{R-E8}) and (\ref{R-E13}) obtained from Theorem \ref{thm1.2}.
Then, there exists a function
$(\mathcal{N},\mathcal{U},\mathcal{E})$ which is a global solution
to the drift-diffusion equations (\ref{R-E9}) satisfying
$$(\mathcal{N},\mathcal{U},\mathcal{E}) \in \mathcal{C}([0,\infty),
B^{\sigma}_{2,1}(\mathbf{R}^{N}))\times L^2([0,\infty),
B^{\sigma}_{2,1}(\mathbf{R}^{N}))\times\mathcal{C}([0,\infty),
B^{\sigma}_{2,1}(\mathbf{R}^{N}))$$ such that as $\tau\rightarrow0$,
it holds that
\begin{eqnarray*}
(n^{\tau},\tau^2\mathbf{u}^{\tau},\sqrt{\tau}\mathbf{E}^{\tau})\rightarrow
(\mathcal{N},\mathbf{0},\mathbf{0})\ \ \ \ \mbox{strongly in}\ \ \
\mathcal{C}([0,T],(B^{\sigma-\delta}_{2,1}(\mathbf{R}^{N}))_{\mathrm{loc}}),
\end{eqnarray*}
\begin{eqnarray*}
\mathbf{u}^{\tau}\rightharpoonup \mathcal{U}\ \ \ \ \mbox{weakly
 in}\ \ \ L^2_{T}(B^{\sigma}_{2,1}(\mathbf{R}^{N})),
\end{eqnarray*}
\begin{eqnarray*}
\nabla\mathbf{B}^{\tau}\rightarrow \mathbf{0}\ \ \ \ \mbox{strongly
 in}\ \ \ L^2_{T}(B^{\sigma-1}_{2,1}(\mathbf{R}^{N})),
\end{eqnarray*}
\begin{eqnarray*}
(\mathbf{E}^{\tau},\mathbf{B}^{\tau})\rightharpoonup
(\mathcal{E},\overline{\mathbf{B}})\ \ \ \ \mbox{weakly$^{\star}$
 in}\ \ \ L^\infty_{T}(B^{\sigma}_{2,1}(\mathbf{R}^{N})),
\end{eqnarray*}
for any $T>0$  and $\delta\in(0,1)$. Moreover, it yields
\begin{eqnarray}
&&\|(\mathcal{N}-\bar{n},
\mathcal{E})(t,\cdot)\|_{B^{\sigma}_{2,1}(\mathbf{R}^{N})}\nonumber\\
&\leq& C_{2}\|(n_{0}-\bar{n}, \mathbf{u}_{0},\mathbf{E}_{0},
\mathbf{B}_{0}-\overline{\mathbf{B}})\|_{B^{\sigma}_{2,1}(\mathbf{R}^{N})},\
t\geq0, \label{R-E14}
\end{eqnarray}
where $C_{2}>0$ is a uniform constant independent of $\tau$.
\end{theorem}

Finally, what left is the combined non-relativistic and relaxation
limits for (\ref{R-E1}). From Theorem \ref{thm1.2} and (\ref{R-E13})
where the superscript $\tau$ is replaced by $(\tau,\varepsilon)$, it
is shown that there exists a unique global in-time classical
solution
$(n^{(\tau,\varepsilon)},\textbf{u}^{(\tau,\varepsilon)},\mathbf{E}^{(\tau,\varepsilon)},\mathbf{B}^{(\tau,\varepsilon)})$
to the system (\ref{R-E10}) and (\ref{R-E13}). Furthermore, we get
\begin{theorem}[\textbf{Combined non-relativistic and relaxation limits}]\label{thm1.5}
Let $(n^{(\tau,\varepsilon)}, \mathbf{u}^{(\tau,\varepsilon)},
\mathbf{E}^{(\tau,\varepsilon)}, \mathbf{B}^{(\tau,\varepsilon)})$
be the global solution of (\ref{R-E10}) and (\ref{R-E13}) obtained
from Theorem \ref{thm1.2}. Then, there exists a function
$(\mathcal{N},\mathcal{U},\mathcal{E})$ which is a global solution
to the drift-diffusion equations (\ref{R-E9}) satisfying
$$(\mathcal{N},\mathcal{U},\mathcal{E}) \in \mathcal{C}([0,\infty),
B^{\sigma}_{2,1}(\mathbf{R}^{N}))\times L^2([0,\infty),
B^{\sigma}_{2,1}(\mathbf{R}^{N}))\times\mathcal{C}([0,\infty),
B^{\sigma}_{2,1}(\mathbf{R}^{N}))$$ such that as $\tau\rightarrow0$
and $\varepsilon\rightarrow0$ simultaneously, it holds that
\begin{eqnarray*}
(n^{(\tau,\varepsilon)},\tau^2\mathbf{u}^{(\tau,\varepsilon)},\sqrt{\tau\varepsilon}\mathbf{E}^{(\tau,\varepsilon)})\rightarrow
(\mathcal{N},\mathbf{0},\mathbf{0})\ \ \ \ \mbox{strongly in}\ \ \
\mathcal{C}([0,T],(B^{\sigma-\delta}_{2,1}(\mathbf{R}^{N}))_{\mathrm{loc}}),
\end{eqnarray*}
\begin{eqnarray*}
\mathbf{u}^{(\tau,\varepsilon)}\rightharpoonup \mathcal{U}\ \ \ \
\mbox{weakly
 in}\ \ \ L^2_{T}(B^{\sigma}_{2,1}(\mathbf{R}^{N})),
\end{eqnarray*}
\begin{eqnarray*}
\nabla\mathbf{B}^{(\tau,\varepsilon)}\rightarrow \mathbf{0}\ \ \ \
\mbox{strongly
 in}\ \ \ L^2_{T}(B^{\sigma-1}_{2,1}(\mathbf{R}^{N})),
\end{eqnarray*}
\begin{eqnarray*}
(\mathbf{E}^{(\tau,\varepsilon)},\mathbf{B}^{(\tau,\varepsilon)})\rightharpoonup
(\mathcal{E},\overline{\mathbf{B}})\ \ \ \ \mbox{weakly$^{\star}$
 in}\ \ \ L^\infty_{T}(B^{\sigma}_{2,1}(\mathbf{R}^{N})),
\end{eqnarray*}
for any $T>0$ and $\delta\in(0,1)$. Moreover, it yields
\begin{eqnarray}
&&\|(\mathcal{N}-\bar{n},
\mathcal{E})(t,\cdot)\|_{B^{\sigma}_{2,1}(\mathbf{R}^{N})}\nonumber\\
&\leq& C_{3}\|(n_{0}-\bar{n}, \mathbf{u}_{0},\mathbf{E}_{0},
\mathbf{B}_{0}-\overline{\mathbf{B}})\|_{B^{\sigma}_{2,1}(\mathbf{R}^{N})},\
t\geq0, \label{R-E15}
\end{eqnarray}
where $C_{3}>0$ is a uniform constant independent of
$(\tau,\varepsilon)$.
\end{theorem}

\begin{remark}\label{rem1.4}
\rm To the best of our knowledge, these limit results (Theorems
\ref{thm1.3}-\ref{thm1.5}) show the convergence \textit{globally} in
time, which have not been appeared in the published literatures. In
comparison with that in \cite{PW1,PW2,PW3}, they hold true in the
functional spaces with relatively \textit{lower} regularity, which
can be regarded as a supplement to the theory of singular limits for
the Euler-Maxwell equations (\ref{R-E1}). Let us mention that the
combined non-relativistic and relaxation limits obtained in Theorem
\ref{thm1.5} does not require any (communication) restriction
between $\tau$ and $\varepsilon$. That is, one can fix any of the
two parameters $\tau$ and $\varepsilon$ and let the other tends to
zero, which is the \textit{genuinely} combined limits.
\end{remark}

\begin{remark}\label{rem1.5}
\rm It is worth noting that Chemin-Lerner's spaces are first
introduced to establish the uniform \textit{a priori} estimates with
respect to $(\tau,\varepsilon)$ and justify the combined limits. As
a matter fact, this approach developed by the current paper can be
applied to study other limit problems with two (or more) independent
singular parameters.
\end{remark}

\begin{remark}\label{rem1.6}
\rm There is no additional conceptual difficulty in considering the
temperature effects and the corresponding balance equation
(\textit{i.e.} non-isentropic Euler-Maxwell equations), although the
estimates are quite tedious.
\end{remark}

The rest of this paper unfolds as follows. In Sect.~\ref{sec:2}, we
introduce the Littlewood-Paley decomposition and recall the
definitions and some useful results on Besov spaces and
Chemin-Lerner's spaces. Sect.~\ref{sec:3} is devoted to the proofs
of main results, which is divided into five subsections for clarity.
In Sect.~\ref{sec:3.1}, we first rewrite the Euler-Maxwell equations
(\ref{R-E1}) as a symmetric hyperbolic system in order to obtain the
effective \textit{a priori} estimate by using Fourier frequency
localization. Furthermore, we give the local existence of classical
solutions in Chemin-Lerner's spaces with critical regularity. Then
in Sect.~\ref{sec:3.2}, we deduce a new uniform \textit{a priori}
estimate under some smallness assumption, which is used to achieve
the (uniform) global existence of classical solutions.
Sect.~\ref{sec:3.3}, Sect.~\ref{sec:3.4} and Sect.~\ref{sec:3.5} are
in turn dedicated to the justification of the non-relativistic
limit, relaxation limit as well as combined non-relativistic and
relaxation limits of Euler-Maxwell equations.

\textbf{Notations}. Throughout the paper, $C$ stands for a uniform
positive constant with respect to $(\tau,\varepsilon)$. The notation
$f\approx g$ means that $f\leq Cg$ and $g\leq Cf$. Denote by
$\mathcal{C}([0,T],X)$ (resp., $\mathcal{C}^{1}([0,T],X)$) the space
of continuous (resp., continuously differentiable) functions on
$[0,T]$ with values in a Banach space $X$. For simplicity, the
notation $\|(f,g,h,k)\|_{X}$ means $
\|f\|_{X}+\|g\|_{X}+\|h\|_{X}+\|k\|_{X}$, where $f,g,h,k\in X$. We
omit the space dependence, since all functional spaces (in $x$) are
considered in $\mathbf{R}^{N}$. Moreover, the integral
$\int_{\mathbf{R}^{N}}fdx$ is labeled as $\int f$ without any
ambiguity.

\section{Tools}\label{sec:2}
The proofs of most of the results presented in this paper require a
dyadic decomposition of Fourier variable. Let us recall briefly the
Littlewood-Paley decomposition theory and the characterization of
Besov spaces and Chemin-Lerner's spaces, see for instance \cite{C,D}
for details.

Let ($\varphi, \chi)$ be a couple of smooth functions valued in [0,
1] such that $\varphi$ is supported in the shell
$\mathcal{C}(0,\frac{3}{4},\frac{8}{3})=\{\xi\in\mathbf{R}^{N}|\frac{3}{4}\leq|\xi|\leq\frac{8}{3}\}$,
$\chi$ is supported in the ball $\mathcal{B}(0,\frac{4}{3})=
\{\xi\in\mathbf{R}^{N}||\xi|\leq\frac{4}{3}\}$ and
$$
\chi(\xi)+\sum_{q=0}^{\infty}\varphi(2^{-q}\xi)=1,\ \ \ \ q\in
\mathbf{Z},\ \  \xi\in\mathbf{R}^{N}.
$$
Let $\mathcal{S'}$ be the dual space of the Schwartz class
$\mathcal{S}$. For $f\in\mathcal{S'}$, the nonhomogeneous dyadic
blocks are defined as follows
$$
\Delta_{-1}f:=\chi(D)f=\tilde{h}\ast f\ \ \ \mbox{with}\ \
\tilde{h}=\mathcal{F}^{-1}\chi,
$$
$$
\Delta_{q}f:=\varphi(2^{-q}D)f=2^{qd}\int h(2^{q}y)f(x-y)dy\ \ \
\mbox{with}\ \ h=\mathcal{F}^{-1}\varphi,\ \ \mbox{if}\ \ q\geq0.
$$
Here $\ast, \ \ \mathcal{F}^{-1} $ represent the convolution
operator and the inverse Fourier transform, respectively. Note that
$\tilde{h}, h\in\mathcal{S}$. The nonhomogeneous Littlewood-Paley
decomposition is
$$
f=\sum_{q \geq-1}\Delta_{q}f \ \ \ \mbox{in}\ \ \ \mathcal{S'}.
$$
Define the low frequency cut-off by
$$
S_{q}f:=\sum_{p\leq q-1}\Delta_{p}f.
$$
According to the above Littlewood-Paley decomposition, thus we
introduce the explicit definition of Besov spaces.
\begin{definition}\label{defn2.1}
Let $1\leq p\leq\infty$ and $s\in \mathbf{R}$. For $1\leq r<\infty$,
the Besov spaces  $B^{s}_{p,r}$ are defined by
$$f\in B^{s}_{p,r} \Leftrightarrow \Big(\sum_{q\geq-1}(2^{qs}\|\Delta_{q}f\|_{L^{p}})^{r}\Big)^{\frac{1}{r}}<\infty$$
and $B^{s}_{p,\infty}$ are defined by
$$f\in B^{s}_{p,\infty} \Leftrightarrow \sup_{q\geq-1}2^{qs}\|\Delta_{q}f\|_{L^{p}}<\infty.$$
\end{definition}
Let us point out that the definition of $B^{s}_{p,r}$ does not
depend on the choice of the Littlewood-Paley decomposition. Now, we
state some classical conclusions, which will be used in subsequent
analysis. The first one is Bernstein's inequality.
\begin{lemma}\label{lem2.1}
Let $k\in\mathbf{N}$ and $0<R_{1}<R_{2}$. There exists a constant
$C$, depending only on $R_{1},R_{2}$ and $N$, such that for all
$1\leq a\leq b\leq\infty$ and $f\in L^{a}$, we have
$$
\mathrm{Supp}\ \mathcal{ F}f\subset
\mathcal{B}(0,R_{1}\lambda)\Rightarrow\sup_{|\alpha|=k}\|\partial^{\alpha}f\|_{L^{b}}\leq
C^{k+1}\lambda^{k+N(\frac{1}{a}-\frac{1}{b})}\|f\|_{L^{a}};$$
$$\mathrm{Supp}\ \mathcal{F}f\subset \mathcal{C}(0,R_{1}\lambda,R_{2}\lambda)\Rightarrow C^{-k-1}\lambda^{k}\|f\|_{L^{a}}\leq \sup_{|\alpha|=k}\|\partial^{\alpha}f\|_{L^{a}}\leq C^{k+1}\lambda^{k}\|f\|_{L^{a}},$$
where $\mathcal{F}f$ represents the Fourier transform on $f$.
\end{lemma}

The second one is a compactness result for Besov spaces.
\begin{proposition}\label{prop2.1}
Let $1\leq p,r\leq \infty,\ s\in \mathbf{R}$ and $\varepsilon>0$.
For all $\phi\in C_{c}^{\infty}$, the map $f\mapsto\phi f$ is
compact from $B^{s+\varepsilon}_{p,r}$ to $B^{s}_{p,r}$.
\end{proposition}

On the other hand, the study of non-stationary partial differential
equations requires spaces of type $L^{\rho}_{T}(X):=L^{\rho}(0,T;X)$
for appropriate Banach spaces $X$. In our case, $X$ is expected to
be a Besov space, so the fundamental idea is to localize the
equations through the Littlewood-Paley decomposition. Then it is
easy to obtain $L^{\rho}_{T}(L^{p})$ estimates for each dyadic
block. Performing a (weighted) $\ell^r$ summation is the most
natural next step. But, in doing so, we get bounds in spaces which
are not type $L^{\rho}_{T}(B^{s}_{p,r})$ (except if $\rho=r$). This
leads to the definition of Chemin-Lerner's spaces first introduced
by J.-Y. Chemin and N. Lerner \cite{C2}, which is the refinement of
the spaces $L^{\rho}_{T}(B^{s}_{p,r})$.

\begin{definition}\label{defn2.2}
For $T>0, s\in\mathbf{R}, 1\leq r,\rho\leq\infty$, set (with the
usual convention if $r=\infty$)
$$\|f\|_{\widetilde{L}^{\rho}_{T}(B^{s}_{p,r})}:
=\Big(\sum_{q\geq-1}(2^{qs}\|\Delta_{q}f\|_{L^{\rho}_{T}(L^{p})})^{r}\Big)^{\frac{1}{r}}.$$
Then we define the space $\widetilde{L}^{\rho}_{T}(B^{s}_{p,r})$ as
the completion of $\mathcal{S}$ over $(0,T)\times\mathbf{R}^{N}$ by
the above norm.
\end{definition}

Furthermore, we define
$$\widetilde{\mathcal{C}}_{T}(B^{s}_{p,r}):=\widetilde{L}^{\infty}_{T}(B^{s}_{p,r})\cap\mathcal{C}([0,T],B^{s}_{p,r})
$$ and $$\widetilde{\mathcal{C}}^1_{T}(B^{s}_{p,r}):=\{f\in\mathcal{C}^1([0,T],B^{s}_{p,r})|\partial_{t}f\in\widetilde{L}^{\infty}_{T}(B^{s}_{p,r})\}.$$
The index $T$ will be omitted when $T=+\infty$. Let us emphasize
that

\begin{remark}\label{rem2.1}
\rm According to Minkowski's inequality, it holds that
$$\|f\|_{\widetilde{L}^{\rho}_{T}(B^{s}_{p,r})}\leq\|f\|_{L^{\rho}_{T}(B^{s}_{p,r})}\,\,\,
\mbox{if}\,\, r\geq\rho,\ \ \ \ \ \
\|f\|_{\widetilde{L}^{\rho}_{T}(B^{s}_{p,r})}\geq\|f\|_{L^{\rho}_{T}(B^{s}_{p,r})}\,\,\,
\mbox{if}\,\, r\leq\rho.
$$\end{remark}
Then, we state the property of continuity for product in
Chemin-Lerner's spaces $\widetilde{L}^{\rho}_{T}(B^{s}_{p,r})$.
\begin{proposition}\label{prop2.2}
The following estimate holds:
$$
\|fg\|_{\widetilde{L}^{\rho}_{T}(B^{s}_{p,r})}\leq
C(\|f\|_{L^{\rho_{1}}_{T}(L^{\infty})}\|g\|_{\widetilde{L}^{\rho_{2}}_{T}(B^{s}_{p,r})}
+\|g\|_{L^{\rho_{3}}_{T}(L^{\infty})}\|f\|_{\widetilde{L}^{\rho_{4}}_{T}(B^{s}_{p,r})})
$$
whenever $s>0, 1\leq p\leq\infty,
1\leq\rho,\rho_{1},\rho_{2},\rho_{3},\rho_{4}\leq\infty$ and
$$\frac{1}{\rho}=\frac{1}{\rho_{1}}+\frac{1}{\rho_{2}}=\frac{1}{\rho_{3}}+\frac{1}{\rho_{4}}.$$
As a direct corollary, one has
$$\|fg\|_{\widetilde{L}^{\rho}_{T}(B^{s}_{p,r})}
\leq
C\|f\|_{\widetilde{L}^{\rho_{1}}_{T}(B^{s}_{p,r})}\|g\|_{\widetilde{L}^{\rho_{2}}_{T}(B^{s}_{p,r})}$$
whenever $s\geq N/p,
\frac{1}{\rho}=\frac{1}{\rho_{1}}+\frac{1}{\rho_{2}}.$
\end{proposition}

In addition, the estimates of commutators in
$\widetilde{L}^{\rho}_{T}(B^{s}_{p,1})$ spaces are also frequently
used in the subsequent analysis. The indices $s,p$ behave just as in
the stationary case \cite{D,FXZ} whereas the time exponent $\rho$
behaves according to H\"{o}lder inequality.
\begin{lemma}\label{lem2.2}
Let $1<p<\infty$ and $1\leq \rho\leq\infty$,  then the following
inequalities are true:
\begin{eqnarray*}
&&2^{qs}\|[f,\Delta_{q}]\mathcal{A}g\|_{L^{\rho}_{T}(L^{p})}\nonumber\\&\leq&
\left\{
\begin{array}{l}
 Cc_{q}\|f\|_{\widetilde{L}^{\rho_{1}}_{T}(B^{s}_{p,1})}\|g\|_{\widetilde{L}^{\rho_{2}}_{T}(B^{s}_{p,1})},\
\ s=1+N/p,\\
 Cc_{q}\|f\|_{\widetilde{L}^{\rho_{1}}_{T}(B^{s}_{p,1})}\|g\|_{\widetilde{L}^{\rho_{2}}_{T}(B^{s+1}_{p,1})},\
 \ s=N/p,\\
 Cc_{q}\|f\|_{\widetilde{L}^{\rho_{1}}_{T}(B^{s+1}_{p,1})}\|g\|_{\widetilde{L}^{\rho_{2}}_{T}(B^{s}_{p,1})},\ \
s=N/p,
\end{array} \right.
\end{eqnarray*}
where the commutator $[\cdot,\cdot]$ is defined by $[f,g]=fg-gf$,
the operator $\mathcal{A}=\mathrm{div}$ or $\mathrm{\nabla}$, $C$ is
a harmless constant, and $c_{q}$ denotes a sequence such that
$\|(c_{q})\|_{ {l^{1}}}\leq
1,\frac{1}{\rho}=\frac{1}{\rho_{1}}+\frac{1}{\rho_{2}}.$
\end{lemma}

Finally, we state a continuity result for compositions (see
\cite{Abidi}) to end up this section.
\begin{proposition}\label{prop2.3}
Let $s>0$, $1\leq p, r, \rho\leq \infty$, $F\in
W^{[s]+1,\infty}_{loc}(I;\mathbf{R})$ with $F(0)=0$, $T\in
(0,\infty]$ and $v\in \widetilde{L}^{\rho}_{T}(B^{s}_{p,r})\cap
L^{\infty}_{T}(L^{\infty}).$ Then
$$\|F(v)\|_{\widetilde{L}^{\rho}_{T}(B^{s}_{p,r})}\leq
C(1+\|v\|_{L^{\infty}_{T}(L^{\infty})})^{[s]+1}\|v\|_{\widetilde{L}^{\rho}_{T}(B^{s}_{p,r})}.$$
\end{proposition}
\section{The proofs of main results}\setcounter{equation}{0} \label{sec:3}
In what follows, we focus on the proofs of main results. For
clarity, we divide them into several subsections, since the proofs
are a bit longer.

\subsection{Reformulation and local existence}\label{sec:3.1}

In this section, we reformulate (\ref{R-E1})-(\ref{R-E3}) in order
to obtain the effective \textit{a priori} estimates by means of
Fourier frequency localization.

For the isentropic case $(\gamma>1)$, let us introduce the sound
speed
$$\psi(n)=\sqrt{P'(n)},$$ and set $\bar{\psi}=\psi(\bar{n})$
corresponding to the sound speed at a background density $\bar{n}$.
Define
\begin{equation}\varrho=\frac{2}{\gamma-1}\Big(\psi(n)-\bar{\psi}\Big), \,\,\, \mathbf{F}=\mathbf{B}-\overline{\mathbf{B}}.\label{R-E16}\end{equation}

Set $$W:=(\varrho,\mathbf{u},\mathbf{E},\mathbf{F})^\top.$$ Then the
system (\ref{R-E1}) can be reduced to the symmetric hyperbolic
system for smooth solutions:
\begin{equation}
\left\{
\begin{array}{l}
\partial_{t}\varrho+\bar{\psi}\mbox{div}\textbf{u}=-\textbf{u}\cdot\nabla
\varrho-\frac{\gamma-1}{2}\varrho\mbox{div}\textbf{u},\\
 \partial_{t}\textbf{u}+\bar{\psi}\nabla \varrho+\frac{\textbf{u}}{\tau}=-\textbf{u}\cdot\nabla\textbf{u}
 -\frac{\gamma-1}{2}\varrho\nabla \varrho-(\mathbf{E}+\varepsilon\textbf{u}\times
 (\mathbf{F}+\overline{\mathbf{B}})),
 \\
 \partial_{t}\mathbf{E}-\frac{1}{\varepsilon}\nabla\times\mathbf{F}=
\bar{n}\textbf{u}+h(\varrho)\textbf{u},\\
\partial_{t}\mathbf{F}+\frac{1}{\varepsilon}\nabla\times\mathbf{E}=0,\\
\nabla\cdot\mathbf{E}=-h(\varrho),\ \ \ \nabla\cdot\mathbf{F}=0,
 \end{array} \right.\label{R-E17}
\end{equation}
where
$h(\varrho)=\{(P_{0}\gamma)^{-\frac{1}{2}}(\frac{\gamma-1}{2}\varrho+\bar{\psi})\}^{\frac{2}{\gamma-1}}-\bar{n}$\
is a smooth function on the domain
$\{\varrho|\frac{\gamma-1}{2}\varrho+\bar{\psi}>0\}$ satisfying
$h(0)=0$. The initial data (\ref{R-E3}) becomes into
\begin{equation}W|_{t=0}=(\varrho_{0},\textbf{u}_{0},\textbf{E}_{0},\textbf{F}_{0})\label{R-E18}\end{equation}
with
$$\varrho_{0}=\frac{2}{\gamma-1}\Big(\psi(n_{0})-\bar{\psi}\Big),
\,\,\, \mathbf{F}_{0}=\mathbf{B}_{0}-\overline{\mathbf{B}}.$$ Under
the symmetrization transform (\ref{R-E16}), the initial data
(\ref{R-E18}) satisfies the corresponding compatible conditions
\begin{equation}
\nabla\cdot\mathbf{E}_{0}=-h(\varrho_{0}),\ \ \
\nabla\cdot\mathbf{F_{0}}=0,\ \ \ x\in \mathbf{R}^{N}. \label{R-E19}
\end{equation}

\begin{remark}\label{rem3.1} \rm The variable change is from the open set $\{(n,\textbf{u},\textbf{E},\textbf{B})\in (0,+\infty)\times \mathbf{R}^{N}\times \mathbf{R}^{N}\times \mathbf{R}^{N}\}$ to the
open set $\{W\in \mathbf{R}\times \mathbf{R}^{N}\times
\mathbf{R}^{N}\times
\mathbf{R}^{N}|\frac{\gamma-1}{2}\varrho+\bar{\psi}>0\}$. It is easy
to show that for classical solutions
$(n,\textbf{u},\textbf{E},\textbf{B})$ away from vacuum,
(\ref{R-E1})-(\ref{R-E3}) is equivalent to
(\ref{R-E17})-(\ref{R-E18}) with
$\frac{\gamma-1}{2}\varrho+\bar{\psi}>0$.
\end{remark}

For the isothermal case where $\gamma=1$, the form of (\ref{R-E17})
is still valid with $\bar\psi = \sqrt{P_{0}}$, while the
symmetrization transform depends on the following enthalpy variable
change
\begin{eqnarray}
\varrho = \sqrt{P_{0}}(\ln n - \ln\bar n)\label{R-E20},
\end{eqnarray}
for the details, see e.g. \cite{FXZ}.

Without loss of generality, we shall study the system
(\ref{R-E17})-(\ref{R-E18}) for $\gamma>1$ and prove main results,
since the case of $\gamma=1$ can be discussed in the same way.

In \cite{K,M}, Kato and Majda established a local existence theory
for generally symmetric hyperbolic systems pertaining to data in the
Sobolev spaces with higher regularity. Recently, using the
regularized means and compactness argument, we have established a
local existence in the framework critical Besov spaces for the
Euler-Poisson equations (\ref{R-E6}), see \cite{FXZ}. In the present
paper, we further strengthen the result such that it holds in
Chemin-Lerner's spaces with critical regularity. Our result reads as
follows.\\

\begin{proposition}\label{prop3.1} For any fixed $0<\tau,\varepsilon\leq1$,
assume that $W_{0}\in{B^{\sigma}_{2,1}}$ satisfying
$\frac{\gamma-1}{2}\varrho_{0}+\bar{\psi}>0$ and (\ref{R-E19}), then
there exist a time $T_{0}>0$ (depending only on the initial data)
and a unique solution $W$ to (\ref{R-E17})-(\ref{R-E18}) such that
$W\in \mathcal{C}^{1}([0,T_{0}]\times \mathbf{R}^{N})$ with
$\frac{\gamma-1}{2}\varrho+\bar{\psi}>0$ for all $t\in[0,T_{0}]$ and
$W\in \widetilde{\mathcal{C}}_{T_{0}}(B^{\sigma}_{2,1})\cap
\widetilde{\mathcal{C}}^1_{T_{0}}(B^{\sigma-1}_{2,1})$.
\end{proposition}
\begin{proof}
Let the assumptions of Proposition \ref{prop3.1} be fulfilled. By a
proper revision, the local result in \cite{FXZ} can also be adapted
to the Euler-Maxwell equations (\ref{R-E17})-(\ref{R-E18}). That is,
there exist a time $T_{0}>0$ (depending only on the initial data)
and a unique solution $W$ to (\ref{R-E17})-(\ref{R-E18}) such that
 $W\in \mathcal{C}^{1}([0,T_{0}]\times
\mathbf{R}^{N})$ with $\frac{\gamma-1}{2}\varrho+\bar{\psi}>0$ for
all $t\in[0,T_{0}]$ and $W \in
\mathcal{C}([0,T_{0}],B^{\sigma}_{2,1})\cap
\mathcal{C}^1([0,T_{0}],B^{\sigma-1}_{2,1})$. In order to prove
Proposition \ref{prop3.1}, it suffices to show that $W\in
\widetilde{L}^{\infty}_{T_{0}}(B^{\sigma}_{2,1})$ and $W_{t}\in
\widetilde{L}^{\infty}_{T_{0}}(B^{\sigma-1}_{2,1})$.

Indeed, applying the operator $\Delta_{q}(q\geq-1)$ to
(\ref{R-E17}), we infer that $(\Delta_{q}\varrho,
\Delta_{q}\textbf{u}, \Delta_{q}\textbf{E},\\ \Delta_{q}\textbf{F})$
satisfies
\begin{equation}
\left\{
\begin{array}{l}
\partial_{t}\Delta_{q}\varrho+\bar{\psi}\Delta_{q}\mbox{div}\textbf{u}=-(\textbf{u}\cdot\nabla)\Delta_{q}\varrho+[\textbf{u},\Delta_{q}]\cdot\nabla
\varrho-\frac{\gamma-1}{2}\Delta_{q}(\varrho\mbox{div}\textbf{u}),\\[1mm]
 \partial_{t}\Delta_{q}\textbf{u}+\bar{\psi}\Delta_{q}\nabla \varrho+\frac{1}{\tau}\Delta_{q}\textbf{u}
 =-(\textbf{u}\cdot\nabla)\Delta_{q}\textbf{u}+[\textbf{u},\Delta_{q}]\cdot\nabla\textbf{u}
 \\\hspace{28mm}-\frac{\gamma-1}{2}\Delta_{q}(\varrho\nabla
 \varrho)-\Delta_{q}\textbf{E}-\varepsilon\Delta_{q}\textbf{u}\times\overline{\mathbf{B}}-\varepsilon\Delta_{q}(\textbf{u}\times\mathbf{F}),\\[1mm]
 \partial_{t}\Delta_{q}\mathbf{E}-\frac{1}{\varepsilon}\nabla\times\Delta_{q}\mathbf{F}=
\bar{n}\Delta_{q}\textbf{u}+\Delta_{q}(h(\varrho)\textbf{u}),\\[1mm]
\partial_{t}\Delta_{q}\mathbf{F}+\frac{1}{\varepsilon}\nabla\times\Delta_{q}\mathbf{E}=0,
\end{array}
\right.\label{R-E21}
\end{equation}
where the commutator $[\cdot,\cdot]$ is defined by $[f,g]=fg-gf$.

Then multiplying the first equation of Eq. (\ref{R-E21}) by
$\Delta_{q}\varrho$, the second one by $\Delta_{q}\textbf{u}$, and
adding the resulting equations together, after integrating it over
$\mathbf{R}^{N}$, we have the energy equality
\begin{eqnarray}&&\frac{1}{2}\frac{d}{dt}\Big(\|\Delta_{q}\varrho\|^2_{L^2}+\|\Delta_{q}\textbf{u}\|^2_{L^2}\Big)+\frac{1}{\tau}\|\Delta_{q}\textbf{u}\|^2_{L^2}\nonumber
\\&=&\frac{1}{2}\int\mathrm{div}\textbf{u}(|\Delta_{q}\varrho|^2+|\Delta_{q}\textbf{u}|^2)+\int([\textbf{u},\Delta_{q}]\cdot\nabla
\varrho\Delta_{q}\varrho+[\textbf{u},\Delta_{q}]\cdot\nabla\textbf{u}\Delta_{q}\textbf{u})
\nonumber\\&&+\frac{\gamma-1}{2}\int\Delta_{q}\varrho(\nabla
\varrho\cdot\Delta_{q}\textbf{u})-\frac{\gamma-1}{2}\int\Big([\Delta_{q},\varrho]\mathrm{div}\textbf{u}\Delta_{q}\varrho+[\Delta_{q},\varrho]\nabla
\varrho\cdot\Delta_{q}\textbf{u}\Big)\nonumber\\&&-\int\Delta_{q}\textbf{E}\cdot\Delta_{q}\textbf{u}-\varepsilon\int\Delta_{q}(\textbf{u}\times\mathbf{F})\cdot\Delta_{q}\textbf{u},
\label{R-E22}\end{eqnarray} where we have used the fact
$\varepsilon(\Delta_{q}\textbf{u}\times\overline{\mathbf{B}})\cdot\Delta_{q}\textbf{u}=0.$

On the other hand, multiplying the third equation of Eq.
(\ref{R-E21}) by $\frac{1}{\bar{n}}\Delta_{q}\textbf{E}$ and the
last one by $\frac{1}{\bar{n}}\Delta_{q}\textbf{F}$, integrating it
over $\mathbf{R}^{N}$ after adding the resulting equations together
implies
\begin{eqnarray}&&\frac{1}{2\bar{n}}\frac{d}{dt}\Big(\|\Delta_{q}\textbf{E}\|^2_{L^2}+\|\Delta_{q}\textbf{F}\|^2_{L^2}\Big)\nonumber
\\&=&\int\Delta_{q}\textbf{u}\cdot\Delta_{q}\textbf{E}+\frac{1}{\bar{n}}\Delta_{q}(h(\varrho)\textbf{u})\cdot\Delta_{q}\textbf{E},\label{R-E23}\end{eqnarray}
where we used the vector analysis formula
$\nabla\cdot(\mathbf{f}\times
\mathbf{g})=(\nabla\times\mathbf{f})\cdot\mathbf{g}-(\nabla\times\mathbf{g})\cdot\mathbf{f}.$

Combining  with the above identities (\ref{R-E22})-(\ref{R-E23}),
with the aid of Cauchy-Schwartz inequality, we get
\begin{eqnarray}
&&\frac{1}{2}\frac{d}{dt}\Big(\|\Delta_{q}\varrho\|^2_{L^2}+\|\Delta_{q}\textbf{u}\|^2_{L^2}+\frac{1}{\bar{n}}\|\Delta_{q}\textbf{E}\|^2_{L^2}
+\frac{1}{\bar{n}}\|\Delta_{q}\textbf{F}\|^2_{L^2}\Big)+\frac{1}{\tau}\|\Delta_{q}\textbf{u}\|^2_{L^2}
\nonumber
\\&\leq&\frac{1}{2}\|\nabla\textbf{u}\|_{L^{\infty}}(\|\Delta_{q}\varrho\|^2_{L^2}+\|\Delta_{q}\textbf{u}\|^2_{L^2})+\frac{\gamma-1}{2}\|\nabla\varrho\|_{L^{\infty}}
\|\Delta_{q}\varrho\|_{L^2}\|\Delta_{q}\textbf{u}\|_{L^2}\nonumber
\\&&+\|[\textbf{u},\Delta_{q}]\nabla\varrho\|_{L^2}\|\Delta_{q}\varrho\|_{L^2}+\|[\textbf{u},\Delta_{q}]\cdot\nabla\textbf{u}\|_{L^2}\|\Delta_{q}\textbf{u}\|_{L^2}
\nonumber
\\&&+\frac{\gamma-1}{2}\|[\varrho,\Delta_{q}]\nabla
\varrho\|_{L^2}\|\Delta_{q}\textbf{u}\|_{L^2}+\frac{\gamma-1}{2}\|[\varrho,\Delta_{q}]\mathrm{div}\textbf{u}\|_{L^2}\|\Delta_{q}\varrho\|_{L^2}
\nonumber
\\&&+\varepsilon\|\Delta_{q}(\textbf{u}\times\mathbf{F})\|_{L^2}\|\Delta_{q}\textbf{u}\|_{L^2}+\frac{1}{\bar{n}}\|\Delta_{q}(h(\varrho)\textbf{u})\|_{L^2}\|\Delta_{q}\textbf{E}\|_{L^2},\ \
t\in[0,T_{0}]. \label{R-E24}
\end{eqnarray}
Next, we may neglect the effect of relaxation term
$\frac{1}{\tau}\|\Delta_{q}\textbf{u}\|^2_{L^2}$, since it is only
responsible for the large time behavior of solutions to
(\ref{R-E17})-(\ref{R-E18}). Dividing (\ref{R-E24}) by
$\{(\|\Delta_{q}\varrho\|^2_{L^2}+\|\Delta_{q}\textbf{u}\|^2_{L^2}+\frac{1}{\bar{n}}\|\Delta_{q}\textbf{E}\|^2_{L^2}
+\frac{1}{\bar{n}}\|\Delta_{q}\textbf{F}\|^2_{L^2})+\epsilon\}^{\frac{1}{2}}$
($\epsilon>0$ a small quantity), we obtain
\begin{eqnarray}
&&\frac{1}{2}\frac{d}{dt}\Big\{\Big(\|\Delta_{q}\varrho\|^2_{L^2}+\|\Delta_{q}\textbf{u}\|^2_{L^2}+\frac{1}{\bar{n}}\|\Delta_{q}\textbf{E}\|^2_{L^2}
+\frac{1}{\bar{n}}\|\Delta_{q}\textbf{F}\|^2_{L^2}\Big)+\epsilon\Big\}^{1/2}
\nonumber
\\&\leq&\frac{1}{2}\|\nabla\textbf{u}\|_{L^{\infty}}(\|\Delta_{q}\varrho\|_{L^2}+\|\Delta_{q}\textbf{u}\|_{L^2})+\frac{\gamma-1}{2}\|\nabla\varrho\|_{L^{\infty}}
\|\Delta_{q}\textbf{u}\|_{L^2}\nonumber
\\&&+\|[\textbf{u},\Delta_{q}]\nabla\varrho\|_{L^2}+\|[\textbf{u},\Delta_{q}]\cdot\nabla\textbf{u}\|_{L^2}
+\frac{\gamma-1}{2}\|[\varrho,\Delta_{q}]\nabla
\varrho\|_{L^2}\nonumber
\\&&+\frac{\gamma-1}{2}\|[\varrho,\Delta_{q}]\mathrm{div}\textbf{u}\|_{L^2}
+\varepsilon\|\Delta_{q}(\textbf{u}\times\mathbf{F})\|_{L^2}+\frac{1}{\bar{n}}\|\Delta_{q}(h(\varrho)\textbf{u})\|_{L^2}\label{R-E25}
\end{eqnarray}
for $t\in[0,T_{0}]$. Integrating (\ref{R-E25}) with respect to the
variable $t$, then taking $\epsilon\rightarrow0$, and using the
estimates of commutators and continuity for the composition in the
stationary case, see \cite{FXZ}, we arrive at
\begin{eqnarray}
&&2^{q\sigma}\|\Delta_{q}W\|_{L^{\infty}_{t}(L^2)} \nonumber\\&\leq&
C2^{q\sigma}\|\Delta_{q}W_{0}\|_{L^2}
+C\int^{t}_{0}c_{q}(\varsigma)\|(\varrho,\textbf{u},\textbf{F})\|^2_{B^{\sigma}_{2,1}}d\varsigma\nonumber\\
&&+C\int^{t}_{0}2^{q\sigma}\|(\nabla
\varrho,\nabla\textbf{u})\|_{L^{\infty}}\|(\Delta_{q}\varrho,\Delta_{q}\textbf{u})\|_{L^2}d\varsigma,
\label{R-E26}
\end{eqnarray}
where $\|c_{q}(t)\|_{\ell^1}\leq1$, for all $t\in[0,T_{0}]$. Next,
summing up (\ref{R-E26}) on $q\geq-1$ gives
\begin{eqnarray}
\|W\|_{\widetilde{L}^{\infty}_{t}(B^{\sigma}_{2,1})} \leq
C\|W_{0}\|_{B^{\sigma}_{2,1}}
+C\int^{t}_{0}\|W(\cdot,\varsigma)\|^2_{B^{\sigma}_{2,1}}d\varsigma,
\ \ t\in[0,T_{0}]. \label{R-E27}
\end{eqnarray}
Then it follows from Remark \ref{rem2.1} and Gronwall's inequality
that
\begin{eqnarray}
W\in\widetilde{L}^{\infty}_{T_{0}}(B^{\sigma}_{2,1}). \label{R-E28}
\end{eqnarray}
Furthermore, it is just a matter of using the equations
(\ref{R-E17}) and Proposition \ref{prop2.2}, we deduce that
\begin{eqnarray}W_{t}\in\widetilde{L}^{\infty}_{T_{0}}(B^{\sigma-1}_{2,1}).
\label{R-E29}
\end{eqnarray}
Hence, the proof of Proposition \ref{prop3.1} is complete.
\end{proof}

\subsection{Uniform \textit{a priori} estimate and global
existence} \label{sec:3.2} In this section, our central task is to
derive a crucial (uniform) \textit{a priori} estimate, which enables
us to achieve the global
existence of classical solutions to (\ref{R-E17})-(\ref{R-E18}).\\

\begin{proposition}\label{prop4.1}
If $W\in \widetilde{\mathcal{C}}_{T}(B^{\sigma}_{2,1})\cap
\widetilde{\mathcal{C}}^1_{T}(B^{\sigma-1}_{2,1})$ is a solution of
(\ref{R-E17})-(\ref{R-E18}) for any $T>0$ and
$0<\tau,\varepsilon\leq1$. There exist some positive constants
$\delta_{1}, \mu_{1}$ and $C_{1}$ independent of
$(\tau,\varepsilon)$ such that for any $T>0$, if
\begin{eqnarray}\|W\|_{\widetilde{L}^\infty_{T}(B^{\sigma}_{2,1})}\leq
\delta_{1},\label{R-E30}\end{eqnarray} then
\begin{eqnarray}&&\|W\|_{\widetilde{L}^\infty_{T}(B^{\sigma}_{2,1})}
\nonumber\\&&+\mu_{1}\Big\{\Big\|\Big(\sqrt{\tau}\varrho,\frac{1}{\sqrt{\tau}}\mathbf{u},\sqrt{\tau\varepsilon}\mathbf{E}\Big)\Big\|_{\widetilde{L}^2_{T}(B^{\sigma}_{2,1})}
+\Big\|\frac{1}{\sqrt{\varepsilon}}\nabla\mathbf{F}\Big\|_{\widetilde{L}^2_{T}(B^{\sigma-1}_{2,1})}\Big\}
\nonumber\\&\leq& C_{1}\|W_{0}\|_{B^{\sigma}_{2,1}}.\label{R-E31}
\end{eqnarray}
\end{proposition}

Having Proposition \ref{prop4.1}, thanks to the standard
continuation argument, we can extend the local-in-time solutions in
Proposition \ref{prop3.1}, and achieve the global existence of
classical solutions to the system (\ref{R-E17})-(\ref{R-E18}), here
we omit details, see e.g. \cite{FXZ}. It follows from Remark
\ref{rem2.1}, Proposition \ref{prop2.2} and the imbedding property
$B^{\sigma}_{2,1}\hookrightarrow \mathcal{C}^1$ that $W\in
\mathcal{C}^{1}([0,\infty)\times \mathbf{R}^{N})$ solves
(\ref{R-E17})-(\ref{R-E18}). The choice of $\delta_{1}$ is
sufficient to ensure $\frac{\gamma-1}{2}\varrho+\bar{\psi}>0$. Then
according to Remark \ref{rem3.1}, we know $(n,\textbf{u},
\textbf{E}, \textbf{B})\in \mathcal{C}^{1}([0,\infty)\times
\mathbf{R}^{N})$ is a solution of (\ref{R-E1})-(\ref{R-E3}) with
$n>0$.\ Furthermore, we arrive at Theorem \ref{thm1.2}.

Actually, the proof of Proposition \ref{prop4.1} is to capture the
dissipation rates from contributions of $(\varrho,\textbf{u},
\textbf{E}, \textbf{F})$ in turn by using the high- and
low-frequency decomposition methods. To do this, we divide it into
several
lemmas.\\

\begin{lemma}\label{lem4.1}
If $W\in \widetilde{\mathcal{C}}_{T}(B^{\sigma}_{2,1})\cap
\widetilde{\mathcal{C}}^1_{T}(B^{\sigma-1}_{2,1})$ is a solution of
(\ref{R-E17})-(\ref{R-E18}) for any $T>0$ and
$0<\tau,\varepsilon\leq1$, then the following estimate holds:
\begin{eqnarray}
&&\|W\|_{\widetilde{L}_{T}^{\infty}(B^{\sigma}_{2,1})}+\sqrt{\frac{\mu_{2}}{\tau}}\|\mathbf{u}\|_{\widetilde{L}^2_{T}(B^{\sigma}_{2,1})}
\nonumber
\\&\leq&C\|W_{0}\|_{B^{\sigma}_{2,1}}+C\sqrt{\|W\|_{\widetilde{L}_{T}^{\infty}(B^{\sigma}_{2,1})}}
\Big\|\Big(\sqrt{\tau}\varrho,\frac{1}{\sqrt{\tau}}\mathbf{u}\Big)\Big\|_{\widetilde{L}_{T}^2(B^{\sigma}_{2,1})},
\label{R-E32}
\end{eqnarray}
where $\mu_{2},C$ are some uniform positive constants independent of
$(\tau,\varepsilon)$.
\end{lemma}

\begin{proof}
By integrating (\ref{R-E24}) with respect to $t\in[0,T]$, with the
help of Cauchy-Schwartz inequality,  we have
\begin{eqnarray}
&&\frac{1}{2}\Big(\|\Delta_{q}\varrho\|^2_{L^2}+\|\Delta_{q}\textbf{u}\|^2_{L^2}+\frac{1}{\bar{n}}\|\Delta_{q}\textbf{E}\|^2_{L^2}
+\frac{1}{\bar{n}}\|\Delta_{q}\textbf{F}\|^2_{L^2}\Big)\Big|^{t}_{0}\nonumber
+\frac{1}{\tau}\|\Delta_{q}\textbf{u}\|^2_{L^2_{t}(L^2)}
\\&\leq&\frac{1}{2}\|\nabla\textbf{u}\|_{L_{T}^2(L^{\infty})}\Big(\|\Delta_{q}\varrho\|_{L^2_{T}(L^2)}\|\Delta_{q}\varrho\|_{L^\infty_{T}(L^2)}
+\|\Delta_{q}\textbf{u}\|_{L^2_{T}(L^2)}\|\Delta_{q}\textbf{u}\|_{L^\infty_{T}(L^2)}\Big)
\nonumber
\\&&+\frac{\gamma-1}{2}\|\nabla\varrho\|_{L_{T}^{\infty}(L^{\infty})}\|\Delta_{q}\varrho\|_{L^2_{T}(L^2)}\|\Delta_{q}\textbf{u}\|_{L^2_{T}(L^2)}
+\|[\textbf{u},\Delta_{q}]\cdot\nabla\varrho\|_{L^2_{T}(L^2)}\|\Delta_{q}\varrho\|_{L^2_{T}(L^2)}\nonumber
\\&&+\|[\textbf{u},\Delta_{q}]\cdot\nabla\textbf{u}\|_{L^2_{T}(L^2)}\|\Delta_{q}\textbf{u}\|_{L^2_{T}(L^2)}
+\frac{\gamma-1}{2}\|[\varrho,\Delta_{q}]\nabla
\varrho\|_{L^2_{T}(L^2)}\|\Delta_{q}\textbf{u}\|_{L^2_{T}(L^2)}\nonumber
\\&&+\frac{\gamma-1}{2}\|[\varrho,\Delta_{q}]\mathrm{div}\textbf{u}\|_{L^2_{T}(L^2)}\|\Delta_{q}\varrho\|_{L^2_{T}(L^2)}
+\varepsilon\|\Delta_{q}(\textbf{u}\times\mathbf{F})\|_{L^2_{T}(L^2)}\|\Delta_{q}\textbf{u}\|_{L^2_{T}(L^2)}
\nonumber
\\&& \hspace{7mm}+\frac{1}{\bar{n}}\|\Delta_{q}(h(\varrho)\textbf{u})\|_{L^1_{T}(L^2)}\|\Delta_{q}\textbf{E}\|_{L^\infty_{T}(L^2)}.\label{R-E33}
\end{eqnarray}
There exists a constant $\mu_{2}>0$ independent of
$(\tau,\varepsilon)$ after multiplying the factor $2^{2q\sigma}$ on
both sides of (\ref{R-E33}), such that
\begin{eqnarray}
&&2^{2q\sigma}\|\Delta_{q}W\|^2_{L^2}+\frac{\mu_{2}}{\tau}2^{2q\sigma}\|\Delta_{q}\textbf{u}\|^2_{L^2_{t}(L^2)}
\nonumber
\\&\leq&C2^{2q\sigma}\|\Delta_{q}W_{0}\|^2_{L^2}\nonumber
\\&&+C\Big\{\|\textbf{u}\|_{\widetilde{L}_{T}^2(B^{\sigma}_{2,1})}\Big(c^2_{q}\|\varrho\|_{\widetilde{L}_{T}^2(B^{\sigma}_{2,1})}\|\varrho\|_{\widetilde{L}_{T}^\infty(B^{\sigma}_{2,1})}+
c^2_{q}\|\textbf{u}\|_{\widetilde{L}_{T}^2(B^{\sigma}_{2,1})}\|\textbf{u}\|_{\widetilde{L}_{T}^\infty(B^{\sigma}_{2,1})}\Big)\nonumber
\\&&+c^2_{q}\|\varrho\|_{\widetilde{L}_{T}^{\infty}(B^{\sigma}_{2,1})}\|\varrho\|_{\widetilde{L}_{T}^2(B^{\sigma}_{2,1})}
\|\textbf{u}\|_{\widetilde{L}_{T}^2(B^{\sigma}_{2,1})}+c^2_{q}\|\textbf{u}\|_{\widetilde{L}_{T}^{\infty}(B^{\sigma}_{2,1})}\|\textbf{u}\|^2_{\widetilde{L}_{T}^2(B^{\sigma}_{2,1})}
\nonumber
\\&&
+c^2_{q}\|\textbf{u}\times\mathbf{F}\|_{\widetilde{L}_{T}^2(B^{\sigma}_{2,1})}\|\textbf{u}\|_{\widetilde{L}_{T}^2(B^{\sigma}_{2,1})}
+c^2_{q}\|h(\varrho)\textbf{u}\|_{\widetilde{L}_{T}^1(B^{\sigma}_{2,1})}\|\textbf{E}\|_{\widetilde{L}_{T}^\infty(B^{\sigma}_{2,1})}\Big\},\label{R-E34}
\end{eqnarray}
where we used Remark \ref{rem2.1}, Lemma \ref{lem2.2} and the
smallness of $\varepsilon(0<\varepsilon\leq1)$; Here and below $C>0$
denotes a uniform constant independent of $(\tau,\varepsilon)$;
$\{c_{q}\}$ denotes some sequence which satisfies $\|(c_{q})\|_{
{l^{1}}}\leq 1$ although each $\{c_{q}\}$ is possibly different in
(\ref{R-E34}).

Then, with aid of Young's inequality($\sqrt{fg}\leq (f+g)/2,\
f,g\geq0$), it follows from Proposition \ref{prop2.2} and
Proposition \ref{prop2.3} that
\begin{eqnarray}
&&2^{q\sigma}\|\Delta_{q}W\|_{L^\infty_{T}(L^2)}+\sqrt{\frac{\mu_{2}}{\tau}}2^{q\sigma}\|\Delta_{q}\textbf{u}\|_{L^2_{T}(L^2)}
\nonumber
\\&\leq&C2^{q\sigma}\|\Delta_{q}W_{0}\|_{L^2}\nonumber
\\&&+Cc_{q}\sqrt{\|\varrho\|_{\widetilde{L}_{T}^{\infty}(B^{\sigma}_{2,1})}}\Big(\sqrt{\tau}\|\varrho\|_{\widetilde{L}_{T}^2(B^{\sigma}_{2,1})}+\frac{1}{\sqrt{\tau}}\|\textbf{u}\|_{\widetilde{L}_{T}^2(B^{\sigma}_{2,1})}\Big)
\nonumber
\\&&+Cc_{q}\sqrt{\|\textbf{u}\|_{\widetilde{L}_{T}^{\infty}(B^{\sigma}_{2,1})}}\frac{1}{\sqrt{\tau}}\|\textbf{u}\|_{\widetilde{L}_{T}^2(B^{\sigma}_{2,1})}
+Cc_{q}\sqrt{\|\textbf{F}\|_{\widetilde{L}_{T}^{\infty}(B^{\sigma}_{2,1})}}
\frac{1}{\sqrt{\tau}}\|\textbf{u}\|_{\widetilde{L}_{T}^2(B^{\sigma}_{2,1})}\nonumber
\\&&+Cc_{q}\sqrt{\|\textbf{E}\|_{\widetilde{L}_{T}^{\infty}(B^{\sigma}_{2,1})}}\Big(\sqrt{\tau}\|\varrho\|_{\widetilde{L}_{T}^2(B^{\sigma}_{2,1})}+\frac{1}{\sqrt{\tau}}\|\textbf{u}\|_{\widetilde{L}_{T}^2(B^{\sigma}_{2,1})}\Big).
\label{R-E35}
\end{eqnarray}
Hence, summing up (\ref{R-E35}) on $q\geq-1$ yields
\begin{eqnarray*}
&&\|W\|_{\widetilde{L}_{T}^{\infty}(B^{\sigma}_{2,1})}+\sqrt{\frac{\mu_{2}}{\tau}}\|\textbf{u}\|_{L^2_{T}(B^{\sigma}_{2,1})}
\nonumber
\\&\leq&C\|W_{0}\|_{B^{\sigma}_{2,1}}+C\sqrt{\|W\|_{\widetilde{L}_{T}^{\infty}(B^{\sigma}_{2,1})}}
\|(\sqrt{\tau}\varrho,\frac{1}{\sqrt{\tau}}\textbf{u})\|_{\widetilde{L}_{T}^2(B^{\sigma}_{2,1})},
\end{eqnarray*}
which is just the desired inequality (\ref{R-E32}).
\end{proof}\\

\begin{lemma}\label{lem4.2}
If $W\in \widetilde{\mathcal{C}}_{T}(B^{\sigma}_{2,1})\cap
\widetilde{\mathcal{C}}^1_{T}(B^{\sigma-1}_{2,1})$ is a solution of
(\ref{R-E17})-(\ref{R-E18}) for any $T>0$ and
$0<\tau,\varepsilon\leq1$, then the following estimate holds:
\begin{eqnarray}
&&\sqrt{\tau}\|\varrho\|_{\widetilde{L}^2_{T}(B^{\sigma}_{2,1})}\nonumber\\&\leq&C(\|(\varrho,\mathbf{u})\|_{\widetilde{L}^\infty_{T}(B^{\sigma}_{2,1})}+\|(\varrho_{0},\mathbf{u}_{0})\|_{B^{\sigma}_{2,1}})
\nonumber\\&&+C\Big\{\frac{1}{\sqrt{\tau}}\|\mathbf{u}\|_{\widetilde{L}^2_{T}(B^{\sigma}_{2,1})}
+\sqrt{\|(\varrho,\mathbf{u},\mathbf{F})\|_{\widetilde{L}^\infty_{T}(B^{\sigma}_{2,1})}}\Big\|\Big(\sqrt{\tau}\varrho,\frac{1}{\sqrt{\tau}}\mathbf{u}\Big)\Big\|_{\widetilde{L}^2_{T}(B^{\sigma}_{2,1})}\Big\}.\label{R-E36}
\end{eqnarray}
where $C$ is a uniform positive constant independent of
$(\tau,\varepsilon)$.
\end{lemma}
\begin{proof}
Set
$$ W_{I}=\left(%
\begin{array}{c}
  \varrho \\
  \textbf{u}\\\end{array}%
  \right),\
A^{I}_{j}(\textbf{u})=\left(%
\begin{array}{cc}
  u^{j} & \bar{\psi}e_{j}^\top \\
  \bar{\psi}e_{j} & u^{j}I_{N\times N} \\
\end{array}%
\right)$$
$$(I_{N\times N}\ \mbox{denotes the unit matrix of order}\  N $$$$\mbox{and}\ e_{j}\ \mbox{is}\ N\mbox{-dimensional
vector where the $j$th component is one, others are zero} ).$$ Then
the first two equations of (\ref{R-E17}) for $W_{I}$ can be written
as the following vector form
\begin{equation}\partial_{t}W_{I}+\sum_{j=1}^{N}A^{I}_{j}(\textbf{u})\partial_{x_{j}}W_{I}=\left(%
\begin{array}{c}
  -\frac{\gamma-1}{2}\varrho\mathrm{div}\textbf{u} \\
   -\frac{\textbf{u}}{\tau}-\frac{\gamma-1}{2}\varrho\nabla \varrho+\mathbf{G} \\
\end{array}%
\right),\label{R-E37}
\end{equation}
where
$\mathbf{G}:=-(\mathbf{E}+\varepsilon\mathbf{u}\times(\mathbf{F}+\overline{\mathbf{B}})).$

To capture the dissipation rate of $\varrho$,  we make the best use
of Shizuta-Kawashima skew-symmetric condition in Fourier spaces,
which was developed for general hyperbolic systems of balance laws
\cite{KY,Y}. Thanks to the isentropic Euler equations (\ref{R-E37}),
 the concrete information of skew-symmetry matrix $K(\xi)$ is well
known (\textit{e.g.} see \cite{CG}), which is very helpful to
estimate the coupled electromagnetic field
$(\mathbf{E},\mathbf{B})$. Now we state the structural condition.\\

\begin{lemma}[Shizuta-Kawashima] \label{lem4.3}For all ~$\xi\in \mathbf{R}^{N},\
\xi\neq0$, there exists a real skew-symmetric smooth matrix $K(\xi)$
which is defined in the unit sphere $\textbf{S}^{N-1}$:
\begin{eqnarray}
K(\xi)=\left(%
\begin{array}{cc}
  0 & \frac{\xi^\top}{|\xi|} \\
  -\frac{\xi}{|\xi|} & 0 \\
\end{array}%
\right),\label{R-E38}
\end{eqnarray}
such that
\begin{eqnarray}
K(\xi)\sum_{j=1}^{N}\xi_{j}A^{I}_{j}(0)=\left(%
\begin{array}{cc}
  \bar{\psi}|\xi| & 0 \\
  0 & -\bar{\psi}\frac{\xi\otimes\xi}{|\xi|} \\
\end{array}%
\right),\label{R-E39}
\end{eqnarray}
where $A^{I}_{j}$ is the matrix appearing in the system
(\ref{R-E37}).
\end{lemma}

First, we rewrite (\ref{R-E37}) into the linearized form
\begin{equation}\partial_{t}W_{I}+\sum_{j=1}^{N}A^{I}_{j}(0)\partial_{x_{j}}W_{I}
=\mathcal{G}+\left(%
\begin{array}{c}
  -\frac{\gamma-1}{2}\varrho\mathrm{div}\textbf{u} \\
   -\frac{\textbf{u}}{\tau}-\frac{\gamma-1}{2}\varrho\nabla \varrho+\textbf{G} \\
\end{array}%
\right),\label{R-E40}\end{equation}where
\begin{eqnarray}
\mathcal{G}=\sum_{j=1}^{N}\Big\{A^{I}_{j}(0)-A^{I}_{j}(\textbf{u})\Big\}\partial_{x_{j}}W_{I}.\label{R-E41}
\end{eqnarray}
Applying the operator $\Delta_{q}$ to the system (\ref{R-E40}) gives
\begin{eqnarray}&&\partial_{t}\Delta_{q}W_{I}+\sum_{j=1}^{N}A^{I}_{j}(0)\partial_{x_{j}}\Delta_{q}W_{I}
\nonumber\\&=&\Delta_{q}\mathcal{G}+\left(%
\begin{array}{c}
  -\frac{\gamma-1}{2}\Delta_{q}(\varrho\mathrm{div}\textbf{u}) \\
   -\frac{\Delta_{q}\textbf{u}}{\tau}-\frac{\gamma-1}{2}\Delta_{q}(\varrho\nabla \varrho)+\Delta_{q}\textbf{G}\\
\end{array}%
\right).\label{R-E42}
\end{eqnarray}
Then we perform the Fourier transform (in the space variable $x$)
for (\ref{R-E42}), multiply the resulting equation by
$-i\tau(\widehat{\Delta_{q}W_{I}})^{\ast}K(\xi)$($^{\ast}$
represents transpose and conjugator), and take the real part of each
term in the equality. Using the expression (\ref{R-E38}) of the
matrix $K(\xi)$ we obtain
\begin{eqnarray}
&& \tau\mathrm{Im}
\Big((\widehat{\Delta_{q}W_{I}})^{\ast}K(\xi)\frac{d}{dt}\widehat{\Delta_{q}W_{I}}\Big)+\tau(\widehat{\Delta_{q}W_{I}})^{\ast}K(\xi)
\Big(\sum_{j=1}^{N}\xi_{j}A^{I}_{j}(0)\Big)\widehat{\Delta_{q}W_{I}}\nonumber\\&=&
\tau\mathrm{Im}\Big((\widehat{\Delta_{q}W_{I}})^{\ast}K(\xi)(\widehat{\Delta_{q}\mathcal{G}})\Big)-\mathrm{Im}
\Big((\overline{\widehat{\Delta_{q}\varrho}})\frac{\xi^{\top}}{|\xi|}\widehat{\Delta_{q}\textbf{u}}\Big)+\tau\mathrm{Im}
\Big((\overline{\widehat{\Delta_{q}\varrho}})\frac{\xi^{\top}}{|\xi|}\widehat{\Delta_{q}\textbf{G}}\Big)
\nonumber\\&&+\frac{\gamma-1}{2}\tau\mathrm{Im}\Big(\overline{\widehat{\Delta_{q}\textbf{u}}}\cdot\frac{\xi}{|\xi|}\widehat{(\Delta_{q}(\varrho\mathrm{div}\textbf{u}))}\Big)
-\frac{\gamma-1}{2}\tau\mathrm{Im}\Big(\overline{\widehat{\Delta_{q}\varrho}}\frac{\xi^{\top}}{|\xi|}\widehat{(\Delta_{q}(\varrho\nabla
\varrho))}\Big).\label{R-E43}
\end{eqnarray}
The skew-symmetry of $K(\xi)$ implies the relation
\begin{eqnarray}
\mathrm{Im}
\Big((\widehat{\Delta_{q}W_{I}})^{\ast}K(\xi)\frac{d}{dt}\widehat{\Delta_{q}W_{I}}\Big)=\frac{1}{2}\frac{d}{dt}\mathrm{Im}
\Big((\widehat{\Delta_{q}W_{I}})^{\ast}K(\xi)\widehat{\Delta_{q}W_{I}}\Big).\label{R-E44}
\end{eqnarray}
Substituting (\ref{R-E39}) into the second term on the left-hand
side of (\ref{R-E43}), it is not difficult to get a lower bound.
Indeed, we have
\begin{eqnarray}
&&\tau\mathrm{Im}
\Big((\widehat{\Delta_{q}W_{I}})^{\ast}K(\xi)\frac{d}{dt}\widehat{\Delta_{q}W_{I}}\Big)+\tau(\widehat{\Delta_{q}W_{I}})^{\ast}K(\xi)
\Big(\sum_{j=1}^{N}\xi_{j}A^{I}_{j}(0)\Big)\widehat{\Delta_{q}W_{I}}\nonumber\\&\geq&
\frac{\tau}{2}\frac{d}{dt}\mathrm{Im}
\Big((\widehat{\Delta_{q}W_{I}})^{\ast}K(\xi)\widehat{\Delta_{q}W_{I}}\Big)+\bar{\psi}\tau|\xi||\widehat{\Delta_{q}W_{I}}|^2-2\bar{\psi}
\tau|\xi||\widehat{\Delta_{q}\textbf{u}}|^2.\label{R-E45}
\end{eqnarray}
With the help of Young inequality and the uniform boundedness of the
matrix $K(\xi)(\xi\neq 0)$, the right-side of (\ref{R-E43}) can be
estimated as
\begin{eqnarray}
&&\tau\mathrm{Im}\Big((\widehat{\Delta_{q}W_{I}})^{\ast}K(\xi)(\widehat{\Delta_{q}\mathcal{G}})\Big)-\mathrm{Im}
\Big((\overline{\widehat{\Delta_{q}\varrho}})\frac{\xi^{\top}}{|\xi|}\widehat{\Delta_{q}\textbf{u}}\Big)+\tau\mathrm{Im}
\Big((\overline{\widehat{\Delta_{q}\varrho}})\frac{\xi^{\top}}{|\xi|}\widehat{\Delta_{q}\textbf{G}}\Big)
\nonumber\\&&+\frac{\gamma-1}{2}\tau\mathrm{Im}\Big(\overline{\widehat{\Delta_{q}\textbf{u}}}\cdot\frac{\xi}{|\xi|}\widehat{(\Delta_{q}(\varrho\mathrm{div}\textbf{u}))}\Big)
-\frac{\gamma-1}{2}\tau\mathrm{Im}\Big(\overline{\widehat{\Delta_{q}\varrho}}\frac{\xi^{\top}}{|\xi|}\widehat{(\Delta_{q}(\varrho\nabla
\varrho))}\Big)\nonumber\\&\leq&
\frac{\bar{\psi}\tau}{2}|\xi||\widehat{\Delta_{q}W_{I}}|^2+\frac{C}{\tau|\xi|}|\widehat{\Delta_{q}\textbf{u}}|^2
+\tau|\widehat{\Delta_{q}W_{I}}||\widehat{\Delta_{q}\mathcal{G}}|+C\tau|\widehat{\Delta_{q}\textbf{u}}||\widehat{(\Delta_{q}(\varrho\mathrm{div}\textbf{u}))}|
\nonumber\\&&+C\tau|\widehat{\Delta_{q}\varrho}||\widehat{(\Delta_{q}(\varrho\nabla
\varrho)})|+\tau\mathrm{Im}
\Big((\overline{\widehat{\Delta_{q}\varrho}})\frac{\xi^{\top}}{|\xi|}\widehat{\Delta_{q}\textbf{G}}\Big),
\label{R-E46}
\end{eqnarray}
where $C>0$ is a constant independent of $(\tau,\varepsilon)$.
Combining the equality (\ref{R-E43}) and the inequality
(\ref{R-E45})-(\ref{R-E46}), we deduce that
\begin{eqnarray}
&&\frac{\bar{\psi}\tau}{2}|\xi||\widehat{\Delta_{q}W_{I}}|^2\nonumber\\&\leq&\frac{C}{\tau}\Big(|\xi|+\frac{1}{|\xi|}\Big)|\widehat{\Delta_{q}\textbf{u}}|^2
+\tau|\widehat{\Delta_{q}W_{I}}||\widehat{\Delta_{q}\mathcal{G}}|\nonumber\\&&+C\tau|\widehat{\Delta_{q}\textbf{u}}||\widehat{(\Delta_{q}(\varrho\mathrm{div}\textbf{u}))}|+C\tau|\widehat{\Delta_{q}\varrho}||\widehat{(\Delta_{q}(\varrho\nabla
\varrho)})|\nonumber\\&&+\tau\mathrm{Im}
\Big((\overline{\widehat{\Delta_{q}\varrho}})\frac{\xi^{\top}}{|\xi|}\widehat{\Delta_{q}\textbf{G}}\Big)-\frac{\tau}{2}\frac{d}{dt}\mathrm{Im}
\Big((\widehat{\Delta_{q}W_{I}})^{\ast}K(\xi)\widehat{\Delta_{q}W_{I}}\Big).\label{R-E47}
\end{eqnarray}
Multiplying (\ref{R-E47}) by $|\xi|$ and integrating it over
$[0,t]\times\mathbf{R}^{N}$, then using Plancherel's theorem yields
\begin{eqnarray}
&&\frac{\bar{\psi}\tau}{2}\int_{0}^{t}\|\Delta_{q}\nabla W_{I}\|^2_{L^2}d\varsigma\nonumber\\
&\leq&\frac{C}{\tau}\int_{0}^{t}(\|\Delta_{q}\textbf{u}\|^2_{L^2}+\|\Delta_{q}\nabla\textbf{u}\|^2_{L^2})d\varsigma+C\tau\int_{0}^{t}\|\Delta_{q}\nabla
W_{I}\|_{L^2}\|\Delta_{q}\mathcal{G}\|_{L^2}d\varsigma
\nonumber\\&&+C\tau\int_{0}^{t}\|\Delta_{q}\nabla\textbf{u}\|_{L^2}\|\Delta_{q}(\varrho\mathrm{div}\textbf{u})\|_{L^2}d\varsigma
+C\tau\int_{0}^{t}\|\Delta_{q}\nabla\varrho\|_{L^2}\|\Delta_{q}(\varrho\nabla
\varrho)\|_{L^2}d\varsigma\nonumber\\&&+\tau\int_{0}^{t}\mathrm{Im}\int
\Big((\overline{\widehat{\Delta_{q}\varrho}})\xi^{\top}\widehat{\Delta_{q}\textbf{G}}\Big)d\xi
d\varsigma-\frac{\tau}{2}\mathrm{Im}\int|\xi|
\Big((\widehat{\Delta_{q}W_{I}})^{\ast}K(\xi)\widehat{\Delta_{q}W_{I}}\Big)d\xi\Big|^t_0.
\label{R-E48}\end{eqnarray} The matrix $K(\xi)$ is uniform bounded
when $\xi\in\mathbf{R}^{N}(\xi\neq 0)$, thus we have
\begin{eqnarray}
&&-\frac{\tau}{2}\mathrm{Im}\int|\xi|
\Big((\widehat{\Delta_{q}W_{I}})^{\ast}K(\xi)\widehat{\Delta_{q}W_{I}}\Big)d\xi\Big|^t_0
\nonumber\\
&\leq&
C\tau\Big(\int(1+|\xi|^2)|\widehat{\Delta_{q}W_{I}(t)}|^2d\xi+\int(1+|\xi|^2)|\widehat{\Delta_{q}W_{I}(0)}|^2d\xi\Big)\nonumber\\
&\leq& C(\|\Delta_{q}W_{I}(t)\|^2_{L^2}+\|\Delta_{q}\nabla
W_{I}(t)\|^2_{L^2}+\|\Delta_{q}W_{I}(0)\|^2_{L^2}+\|\Delta_{q}\nabla
W_{I}(0)\|^2_{L^2}),\label{R-E49}
\end{eqnarray}
where we used the smallness of $\tau (0<\tau\leq1)$ in the last
step.

Next we turn to estimate the coupled electromagnetic field
$(\mathbf{E},\mathbf{B})$:
\begin{eqnarray}
&&\tau\int_{0}^{t}\mathrm{Im}\int
\Big((\overline{\widehat{\Delta_{q}\varrho}})\xi^{\top}\widehat{\Delta_{q}\textbf{G}}\Big)d\xi
d\varsigma\nonumber\\&=&-\tau\int_{0}^{t}\mathrm{Im}\int\Big((\overline{\widehat{\Delta_{q}\varrho}})\xi^{\top}\widehat{\Delta_{q}\textbf{E}}\Big)d\xi
d\varsigma\nonumber\\&&-\tau\varepsilon\int_{0}^{t}\mathrm{Im}\int
\Big((\overline{\widehat{\Delta_{q}\varrho}})\xi^{\top}\widehat{\Delta_{q}(\textbf{u}\times(\textbf{F}+\overline{\mathbf{B}}))}\Big)d\xi
d\varsigma
\nonumber\\&\equiv&\mathbf{I}_{1}+\mathbf{I}_{2},\label{R-E50}
\end{eqnarray}
where the first term $\mathbf{I}_{1}$ is estimated as follows
\begin{eqnarray}
\mathbf{I}_{1}&=&-\tau\int_{0}^{t}\mathrm{Im}\int\Big((\overline{\widehat{\Delta_{q}\varrho}})\xi^{\top}\widehat{\Delta_{q}\textbf{E}}\Big)d\xi
d\varsigma\nonumber\\&=&-\tau\Big\{-\int^{t}_{0}\frac{\mathrm{\emph{i}}}{2}\int\Big((\overline{\widehat{\Delta_{q}\varrho}})\xi^{\top}\widehat{\Delta_{q}\textbf{E}}\Big)d\xi
d\varsigma
+\int^{t}_{0}\frac{\mathrm{\emph{i}}}{2}\int\Big((\widehat{\Delta_{q}\varrho})\xi^{\top}\overline{\widehat{\Delta_{q}\textbf{E}}}\Big)d\xi d\varsigma\Big\}\nonumber\\
&=&-\frac{\tau}{2}\Big\{\int^{t}_{0}\int(\overline{\widehat{\Delta_{q}\nabla
\varrho}})\cdot\widehat{\Delta_{q}\textbf{E}}d\xi
d\varsigma+\int^{t}_{0}\int(\widehat{\Delta_{q}\nabla
\varrho})\cdot\overline{\widehat{\Delta_{q}\textbf{E}}}d\xi d\varsigma\Big\}\nonumber\\
 &=&-\tau(2\pi)^{N}\int^{t}_{0}\int\Delta_{q}\nabla
 \varrho\cdot\Delta_{q}\textbf{E}dxd\varsigma
\nonumber\\
 &=&\tau(2\pi)^{N}\int^{t}_{0}\int\Delta_{q}\varrho\Delta_{q}\mathrm{div}\textbf{E}dxd\varsigma\nonumber\\
&=&-\tau(2\pi)^{N}\int^{t}_{0}\int\Delta_{q}\varrho\Delta_{q}(h(\varrho)-h(0))dxd\varsigma\nonumber\\
&=&-\tau(P_{0}\gamma)^{-\frac{1}{2}}\bar{n}^{\frac{3-\gamma}{2}}(2\pi)^{N}\int^{t}_{0}\|\Delta_{q}\varrho\|^2_{L^2}d\varsigma-\tau(2\pi)^{N}\int^{t}_{0}\int\Delta_{q}\varrho\Delta_{q}(\tilde{h}(\varrho)\varrho)dxd\varsigma
\nonumber\\ \hspace{10mm}&\leq&
-\tau(P_{0}\gamma)^{-\frac{1}{2}}\bar{n}^{\frac{3-\gamma}{2}}(2\pi)^{N}\int^{t}_{0}\|\Delta_{q}\varrho\|^2_{L^2}d\varsigma+C\tau\int^{t}_{0}\|\Delta_{q}(\tilde{h}(\varrho)\varrho)\|_{L^2}\|\Delta_{q}\varrho\|_{L^2}d\varsigma.
\label{R-E51}\end{eqnarray} Here,
$\tilde{h}(\varrho)=\int_{0}^{1}h'(\epsilon
\varrho)d\varsigma-(P_{0}\gamma)^{-\frac{1}{2}}\bar{n}^{\frac{3-\gamma}{2}}$
is a smooth function on $\{\varrho|\frac{\gamma-1}{2}\epsilon
\varrho+\bar{\psi}>0,\ \epsilon\in[0,1]\}$ satisfying $
\tilde{h}(0)=0$.

In a similar way, $\mathbf{I}_{2}$ is estimated as
\begin{eqnarray}
\mathbf{I}_{2}&=&-\tau\varepsilon\int_{0}^{t}\mathrm{Im}\int
\Big((\overline{\widehat{\Delta_{q}\varrho}})\xi^{\top}\widehat{\Delta_{q}(\textbf{u}\times(\textbf{F}+\overline{\mathbf{B}}))}\Big)d\xi
d\varsigma\nonumber\\
&\leq&
C\tau\varepsilon\int^{t}_{0}\Big(\|\Delta_{q}\mathrm{div}(\textbf{u}\times\overline{\mathbf{B}})\|_{L^2}+\|\Delta_{q}\mathrm{div}(\textbf{u}\times\mathbf{F})\|_{L^2}\Big)\|\Delta_{q}\varrho\|_{L^2}d\varsigma.\label{R-E52}
\end{eqnarray}
Thus, combining with (\ref{R-E48})-(\ref{R-E52}), we get
\begin{eqnarray}
&&\frac{\bar{\psi}\tau}{2}\int^t_0\|\Delta_{q}\nabla W_{I}\|^2_{L^2}d\varsigma+\tau(P_{0}\gamma)^{-\frac{1}{2}}\bar{n}^{\frac{3-\gamma}{2}}(2\pi)^{N}\int^t_0\|\Delta_{q}\varrho\|^2_{L^2}d\varsigma\nonumber\\
&\leq&C(\|\Delta_{q}W_{I}(t)\|^2_{L^2}+\|\Delta_{q}\nabla
W_{I}(t)\|^2_{L^2}+\|\Delta_{q}W_{I}(0)\|^2_{L^2}+\|\Delta_{q}\nabla
W_{I}(0)\|^2_{L^2})
\nonumber\\&&+\frac{C}{\tau}\int_{0}^{t}(\|\Delta_{q}\textbf{u}\|^2_{L^2}+\|\Delta_{q}\nabla\textbf{u}\|^2_{L^2})d\varsigma+C\tau\int_{0}^{t}\|\Delta_{q}\nabla
W_{I}\|_{L^2}\|\Delta_{q}\mathcal{G}\|_{L^2}d\varsigma \nonumber\\&&
+C\tau\int_{0}^{t}\|\Delta_{q}\nabla\textbf{u}\|_{L^2}\|\Delta_{q}(\varrho\mathrm{div}\textbf{u})\|_{L^2}d\varsigma
+C\tau\int_{0}^{t}\|\Delta_{q}\nabla\varrho\|_{L^2}\|\Delta_{q}(\varrho\nabla
\varrho)\|_{L^2}d\varsigma\nonumber\\&&+C\tau\int^{t}_{0}\|\Delta_{q}(\tilde{h}(\varrho)\varrho)\|_{L^2}\|\Delta_{q}\varrho\|_{L^2}d\varsigma\nonumber\\&&
+C\tau\varepsilon\int^{t}_{0}\Big(\|\Delta_{q}\mathrm{div}(\textbf{u}\times\overline{\mathbf{B}})\|_{L^2}
+\|\Delta_{q}\mathrm{div}(\textbf{u}\times\mathbf{F})\|_{L^2}\Big)\|\Delta_{q}\varrho\|_{L^2}d\varsigma.\label{R-E53}\end{eqnarray}
Recalling Lemma \ref{lem2.1}, we have
$$\|\Delta_{q}\nabla f\|_{L^2}\approx
2^{q}\|\Delta_{q}f\|_{L^2}\ (q\geq0).$$ Note that this fact, from
(\ref{R-E53}), we get the high-frequency part of
$\|\Delta_{q}\varrho\|_{L^2_{t}(L^2)}(q\geq0)$:
\begin{eqnarray}
&&\frac{\bar{\psi}\tau}{2}2^{2q}\|\Delta_{q}\varrho\|^2_{L^2_{t}(L^2)}\nonumber\\
&\leq&C(2^{2q}\|\Delta_{q}W_{I}\|^2_{L^\infty_{T}(L^2)}+2^{2q}\|\Delta_{q}W_{I}(0)\|^2_{L^2})
+C\Big\{\frac{2^{2q}}{\tau}\|\Delta_{q}\textbf{u}\|^2_{L^2_{T}(L^2)}\nonumber\\&&\hspace{3mm}+\tau\varepsilon\|\Delta_{q}\varrho\|_{L^2_{T}(L^2)}
\|\Delta_{q}\mathrm{div}(\textbf{u}\times\overline{\mathbf{B}})\|_{L^2_{T}(L^2)}+2^{q}\tau\|\Delta_{q}
W_{I}\|_{L^2_{T}(L^2)}\|\Delta_{q}\mathcal{G}\|_{L^2_{T}(L^2)}
\nonumber\\&&\hspace{3mm}+2^{q}\tau\|\Delta_{q}\textbf{u}\|_{L^2_{T}(L^2)}\|\Delta_{q}(\varrho\mathrm{div}\textbf{u})\|_{L^2_{T}(L^2)}
+2^{q}\tau\|\Delta_{q}\varrho\|_{L^2_{T}(L^2)}\|\Delta_{q}(\varrho\nabla
\varrho)\|_{L^2_{T}(L^2)}
\nonumber\\&&\hspace{3mm}+\tau\|\Delta_{q}(\tilde{h}(\varrho)\varrho)\|_{L^2_{T}(L^2)}\|\Delta_{q}\varrho\|_{L^2_{T}(L^2)}
+\tau\varepsilon\|\Delta_{q}\mathrm{div}(\textbf{u}\times\mathbf{F})\|_{L^2_{T}(L^2)}\|\Delta_{q}\varrho\|_{L^2_{T}(L^2)}\Big\}\label{R-E54}
\end{eqnarray}
and the corresponding low-frequency part:
\begin{eqnarray}
&&\tau(P_{0}\gamma)^{-\frac{1}{2}}\bar{n}^{\frac{3-\gamma}{2}}(2\pi)^{N}\|\Delta_{-1}\varrho\|^2_{L^2_{t}(L^2)}\nonumber\\
&\leq&C(\|\Delta_{-1}W_{I}\|^2_{L^\infty_{T}(L^2)}+\|\Delta_{-1}W_{I}(0)\|^2_{L^2})
\nonumber\\&&+C\Big\{\frac{1}{\tau}\|\Delta_{-1}\textbf{u}\|^2_{L^2_{T}(L^2)}+\tau\varepsilon\|\Delta_{-1}\varrho\|_{L^2_{T}(L^2)}
\|\Delta_{-1}\mathrm{div}(\textbf{u}\times\overline{\mathbf{B}})\|_{L^2_{T}(L^2)}\nonumber\\&&+\tau\|\Delta_{-1}
W_{I}\|_{L^2_{T}(L^2)}\|\Delta_{-1}\mathcal{G}\|_{L^2_{T}(L^2)}
+\tau\|\Delta_{-1}\textbf{u}\|_{L^2_{T}(L^2)}\|\Delta_{-1}(\varrho\mathrm{div}\textbf{u})\|_{L^2_{T}(L^2)}
\nonumber\\&&+\tau\|\Delta_{-1}\varrho\|_{L^2_{T}(L^2)}\|\Delta_{-1}(\varrho\nabla
\varrho)\|_{L^2_{T}(L^2)}
+\tau\|\Delta_{-1}(\tilde{h}(\varrho)\varrho)\|_{L^2_{T}(L^2)}\|\Delta_{-1}\varrho\|_{L^2_{T}(L^2)}
\nonumber\\&&+\tau\varepsilon\|\Delta_{-1}\mathrm{div}(\textbf{u}\times\mathbf{F})\|_{L^2_{T}(L^2)}\|\Delta_{-1}\varrho\|_{L^2_{T}(L^2)}\Big\}.\label{R-E55}
\end{eqnarray}
To conclude, we combine (\ref{R-E54})-(\ref{R-E55}) and multiply the
factor $2^{2q(\sigma-1)}$ on both sides of the resulting inequality
to obtain
\begin{eqnarray}
&&\tau2^{2q\sigma}\|\Delta_{q}\varrho\|^2_{L^2_{t}(L^2)}\nonumber\\
&\leq&Cc^2_{q}(\|W_{I}\|^2_{\widetilde{L}^\infty_{T}(B^{\sigma}_{2,1})}+\|W_{I}(0)\|^2_{B^{\sigma}_{2,1}})
\nonumber\\&&+C\Big\{\frac{c^2_{q}}{\tau}\|\textbf{u}\|^2_{\widetilde{L}^2_{T}(B^{\sigma}_{2,1})}+\tau\varepsilon
c^2_{q}\|\varrho\|_{\widetilde{L}^2_{T}(B^{\sigma}_{2,1})}\|\mathrm{div}(\textbf{u}\times\overline{\mathbf{B}})\|_{\widetilde{L}^2_{T}(B^{\sigma-1}_{2,1})}
\nonumber\\&&+\tau
c^2_{q}\|W_{I}\|_{\widetilde{L}^2_{T}(B^{\sigma}_{2,1})}\|\mathcal{G}\|_{\widetilde{L}^2_{T}(B^{\sigma-1}_{2,1})}
+\tau
c^2_{q}\|\textbf{u}\|_{\widetilde{L}^2_{T}(B^{\sigma}_{2,1})}\|\varrho\mathrm{div}\textbf{u}\|_{\widetilde{L}^2_{T}(B^{\sigma-1}_{2,1})}
\nonumber\\&&+\tau
c^2_{q}\|\varrho\|_{\widetilde{L}^2_{T}(B^{\sigma}_{2,1})}\Big(\|\varrho\nabla
\varrho\|_{\widetilde{L}^2_{T}(B^{\sigma-1}_{2,1})}
+\|\tilde{h}(\varrho)\varrho\|_{\widetilde{L}^2_{T}(B^{\sigma-1}_{2,1})}\nonumber\\&&+\varepsilon\|\mathrm{div}(\textbf{u}\times\mathbf{F})\|_{\widetilde{L}^2_{T}(B^{\sigma-1}_{2,1})}\Big)\Big\},\label{R-E56}
\end{eqnarray}
where $\{c_{q}\}$ denotes some sequence which satisfies
$\|(c_{q})\|_{ {l^{1}}}\leq 1$.

By employing Young's inequality, we are led to the estimate
\begin{eqnarray}
&&\sqrt{\tau}2^{q\sigma}\|\Delta_{q}\varrho\|_{L^2_{T}(L^2)}\nonumber\\&\leq&Cc_{q}(\|W_{I}\|_{\widetilde{L}^\infty_{T}(B^{\sigma}_{2,1})}+\|W_{I}(0)\|_{B^{\sigma}_{2,1}})
\nonumber\\&&+C\Big\{\frac{c_{q}}{\sqrt{\tau}}\|\textbf{u}\|_{\widetilde{L}^2_{T}(B^{\sigma}_{2,1})}+\sqrt{\tau\varepsilon}c_{q}\sqrt{\|\varrho\|_{\widetilde{L}^2_{T}(B^{\sigma}_{2,1})}\|\textbf{u}\|_{\widetilde{L}^2_{T}(B^{\sigma}_{2,1})}}
\nonumber\\&&+c_{q}\sqrt{\|\textbf{u}\|_{\widetilde{L}^\infty_{T}(B^{\sigma}_{2,1})}}\Big(\sqrt{\tau}\|\varrho\|_{\widetilde{L}^2_{T}(B^{\sigma}_{2,1})}+\frac{1}{\sqrt{\tau}}\|\textbf{u}\|_{\widetilde{L}^2_{T}(B^{\sigma}_{2,1})}\Big)
\nonumber\\&&+\frac{c_{q}}{\sqrt{\tau}}\sqrt{\|\varrho\|_{\widetilde{L}^\infty_{T}(B^{\sigma}_{2,1})}}\|\textbf{u}\|_{\widetilde{L}^2_{T}(B^{\sigma}_{2,1})}
\nonumber\\&&+c_{q}\sqrt{\|(\varrho,\mathbf{F})\|_{\widetilde{L}^\infty_{T}(B^{\sigma}_{2,1})}}\Big(\sqrt{\tau}\|\varrho\|_{\widetilde{L}^2_{T}(B^{\sigma}_{2,1})}+\frac{1}{\sqrt{\tau}}\|\textbf{u}\|_{\widetilde{L}^2_{T}(B^{\sigma}_{2,1})}\Big)\Big\}.\label{R-E57}
\end{eqnarray}
In the end, with the help of the smallness of $(\tau,\varepsilon)$,
summing up (\ref{R-E57}) on $q\geq-1$ concludes the inequality
(\ref{R-E36}) readily.
\end{proof}\\

\begin{lemma}\label{lem4.4}
If $W\in \widetilde{\mathcal{C}}_{T}(B^{\sigma}_{2,1})\cap
\widetilde{\mathcal{C}}^1_{T}(B^{\sigma-1}_{2,1})$ is a solution of
(\ref{R-E17})-(\ref{R-E18}) for any $T>0$ and
$0<\tau,\varepsilon\leq1$, then the following estimate holds:
\begin{eqnarray}
&&\sqrt{\tau\varepsilon}\|\mathbf{E}\|_{\widetilde{L}^2_{T}(B^{\sigma}_{2,1})}\nonumber\\
&\leq&C(\|(\mathbf{u},\mathbf{E},\mathbf{F})\|_{\widetilde{L}^{\infty}_{T}(B^{\sigma}_{2,1})}+\|(\mathbf{u}_{0},\mathbf{E}_{0},\mathbf{F}_{0})\|_{B^{\sigma}_{2,1}})
\nonumber\\&&+C\Big\{\Big\|\Big(\sqrt{\tau}\varrho,\frac{\mathbf{u}}{\sqrt{\tau}}\Big)\Big\|_{\widetilde{L}^2_{T}(B^{\sigma}_{2,1})}
+\Big\|\frac{\nabla\mathbf{F}}{\sqrt{\varepsilon}}\Big\|_{\widetilde{L}^2_{T}(B^{\sigma-1}_{2,1})}
\nonumber\\&&
+\sqrt{\|(\varrho,\mathbf{u},\mathbf{F})\|_{\widetilde{L}^\infty_{T}(B^{\sigma}_{2,1})}}\Big[\Big\|\Big(\sqrt{\tau}\varrho,\frac{\mathbf{u}}{\sqrt{\tau}},\sqrt{\tau\varepsilon}\mathbf{E}\Big)\Big\|_{\widetilde{L}^2_{T}(B^{\sigma}_{2,1})}
+\Big\|\frac{\nabla\mathbf{F}}{\sqrt{\varepsilon}}\Big\|_{\widetilde{L}^2_{T}(B^{\sigma-1}_{2,1})}\Big]
\Big\},\label{R-E58}
\end{eqnarray}
where $C>0$ is a uniform constant independent of
$(\tau,\varepsilon)$.
\end{lemma}

\begin{proof}
A nice ``div-curl" construction of Maxwell's equations of
(\ref{R-E17}) enables us to obtain the high-frequency part of
$\mathbf{E}$. Indeed, by applying $\Delta_{q}$ to both side of
$\nabla\cdot\textbf{E}=-h(\varrho)(q\geq0)$, integrating it over
$\mathbf{R}^{N}$ after multiplying
$\nabla\cdot\Delta_{q}\textbf{E}$, in virtue of H\"{o}lder's
inequality, we obtain
\begin{eqnarray}
&&\|\nabla\cdot\Delta_{q}\textbf{E}\|^2_{L^2}\nonumber\\&\leq&
C\{(P_{0}\gamma)^{-\frac{1}{2}}\bar{n}^{\frac{3-\gamma}{2}}\|\Delta_{q}\varrho\|_{L^2}+\|\Delta_{q}(\widetilde{h}(\varrho)\varrho)\|_{L^2}\}\|\nabla\cdot\Delta_{q}\textbf{E}\|_{L^2},\label{R-E59}
\end{eqnarray}
where the function $\widetilde{h}(\varrho)$ is defined by
(\ref{R-E51}).

On the other hand, applying $\Delta_{q}(q\geq0)$ to the fourth
equation of (\ref{R-E17}) and multiplying the resulting equation by
$\nabla\times\Delta_{q}\textbf{E}$, after integration by parts,
yields
\begin{eqnarray}
\|\nabla\times\Delta_{q}\textbf{E}\|^2_{L^2}&=&-\varepsilon\int\partial_{t}\Delta_{q}\textbf{F}\cdot(\nabla\times\Delta_{q}\textbf{E})\nonumber
\\&=&\varepsilon\int(\nabla\times\partial_{t}\Delta_{q}\textbf{F})\cdot\Delta_{q}\textbf{E}. \label{R-E60}
\end{eqnarray}
Substituting the third equation of (\ref{R-E17}) into (\ref{R-E60}),
by Cauchy-Schwartz inequality, leads to
\begin{eqnarray}
&&\|\nabla\times\Delta_{q}\textbf{E}\|^2_{L^2}+\|\nabla\times\Delta_{q}\textbf{F}\|^2_{L^2}\nonumber
\\&\leq&\varepsilon\frac{d}{dt}\int
(\nabla\times\Delta_{q}\mathbf{F})\cdot\Delta_{q}\mathbf{E}+\bar{n}\varepsilon\|\Delta_{q}\mathbf{u}\|_{L^2}\|\Delta_{q}(\nabla\times\mathbf{F})\|_{L^2}
\nonumber
\\&&+\varepsilon\|\Delta_{q}(h(\varrho)\mathbf{u})\|_{L^2}\|\Delta_{q}(\nabla\times\mathbf{F})\|_{L^2}.\label{R-E61}
\end{eqnarray}
Combining with (\ref{R-E59}) and (\ref{R-E61}), it follows from the
elementary relation $$\|\nabla \mathbf{f}\|_{L^2}\approx
\|\nabla\cdot\mathbf{f}\|_{L^2}+\|\nabla\times\mathbf{f}\|_{L^2}$$
that
\begin{eqnarray}
&&\|\nabla\Delta_{q}\textbf{E}\|^2_{L^2}\nonumber
\\&\leq&
C\{(P_{0}\gamma)^{-\frac{1}{2}}\bar{n}^{\frac{3-\gamma}{2}}\|\Delta_{q}\varrho\|_{L^2}+\|\Delta_{q}(\widetilde{h}(\varrho)\varrho)\|_{L^2}\}\|\nabla\cdot\Delta_{q}\textbf{E}\|_{L^2}\nonumber
\\&&+\varepsilon\frac{d}{dt}\int
(\nabla\times\Delta_{q}\mathbf{F})\cdot\Delta_{q}\mathbf{E}+\bar{n}\varepsilon\|\Delta_{q}\mathbf{u}\|_{L^2}\|\Delta_{q}(\nabla\times\mathbf{F})\|_{L^2}
\nonumber
\\&&+\varepsilon\|\Delta_{q}(h(\varrho)\mathbf{u})\|_{L^2}\|\Delta_{q}(\nabla\times\mathbf{F})\|_{L^2}.\label{R-E62}
\end{eqnarray}
Note that $q\geq0$, from Lemma \ref{lem2.1}, we further get
\begin{eqnarray}
&&\tau\varepsilon2^{2q}\|\Delta_{q}\textbf{E}\|^2_{L^2}\nonumber
\\&\leq&
C\tau\|\Delta_{q}\varrho\|^2_{L^2}+C\tau\varepsilon\|\Delta_{q}(\widetilde{h}(\varrho)\varrho)\|_{L^2}2^{q}\|\Delta_{q}\textbf{E}\|_{L^2}
\nonumber
\\&&+\tau\varepsilon^2\frac{d}{dt}\int
(\nabla\times\Delta_{q}\mathbf{F})\cdot\Delta_{q}\mathbf{E}+C\tau\varepsilon^2\|\Delta_{q}\mathbf{u}\|_{L^2}\|\Delta_{q}(\nabla\times\mathbf{F})\|_{L^2}
\nonumber
\\&&+\tau\varepsilon^2\|\Delta_{q}(h(\varrho)\mathbf{u})\|_{L^2}\|\Delta_{q}(\nabla\times\mathbf{F})\|_{L^2},\label{R-E63}
\end{eqnarray}
where the smallness of $\varepsilon$ is used. Integrating
(\ref{R-E63}) in $t\in [0,T]$ implies
\begin{eqnarray}
&&\tau\varepsilon2^{2q}\|\Delta_{q}\textbf{E}\|^2_{L_{t}^2(L^2)}\nonumber
\\&\leq&
\tau\varepsilon^22^{q}\|\Delta_{q}\textbf{F}\|_{L_{T}^\infty(L^2)}\|\Delta_{q}\textbf{E}\|_{L_{T}^\infty(L^2)}
+\tau\varepsilon^22^{q}\|\Delta_{q}\textbf{F}_{0}\|_{L^2}\|\Delta_{q}\textbf{E}_{0}\|_{L^2}\nonumber
\\&&+
C\tau\|\Delta_{q}\varrho\|^2_{L_{T}^2(L^2)}+C\tau\varepsilon\|\Delta_{q}(\widetilde{h}(\varrho)\varrho)\|_{L_{T}^2(L^2)}2^{q}\|\Delta_{q}\textbf{E}\|_{L_{T}^2(L^2)}
\nonumber
\\&&+C\tau\varepsilon^2\|\Delta_{q}\mathbf{u}\|_{L_{T}^2(L^2)}\|\Delta_{q}(\nabla\times\mathbf{F})\|_{L_{T}^2(L^2)}
\nonumber
\\&&+\tau\varepsilon^2\|\Delta_{q}(h(\varrho)\mathbf{u})\|_{L_{T}^2(L^2)}\|\Delta_{q}(\nabla\times\mathbf{F})\|_{L_{T}^2(L^2)}.\label{R-E64}
\end{eqnarray}

On the other hand, the desired low-frequency of $\textbf{E}$ can be
deduced from the Lorentz field in the Euler equations of
(\ref{R-E17}). Using the second equation of (\ref{R-E17}), we have
\begin{eqnarray}
\textbf{u}_{t}+\mathbf{E}=-\bar{\psi}\nabla
\varrho-\frac{\textbf{u}}{\tau}-\textbf{u}\cdot\nabla\textbf{u}
 -\frac{\gamma-1}{2}\varrho\nabla \varrho-\varepsilon\textbf{u}\times
 (\mathbf{F}+\overline{\mathbf{B}}).\label{R-E65}
\end{eqnarray}
Applying the operator $\Delta_{-1}$ to (\ref{R-E65}) implies
\begin{eqnarray}
\partial_{t}\Delta_{-1}\textbf{u}+\Delta_{-1}\mathbf{E}&=&-\bar{\psi}\Delta_{-1}\nabla
\varrho-\frac{\Delta_{-1}\textbf{u}}{\tau}-\Delta_{-1}(\textbf{u}\cdot\nabla\textbf{u})
 \nonumber
\\&&-\frac{\gamma-1}{2}\Delta_{-1}(\varrho\nabla \varrho)-\varepsilon\Delta_{-1}(\textbf{u}\times
 (\mathbf{F}+\overline{\mathbf{B}})).\label{R-E66}
\end{eqnarray}
Multiplying (\ref{R-E66}) by $\tau\varepsilon\Delta_{-1}\mathbf{E}$
and integrating the resulting equation over $\mathbf{R}^{N}$, we get
\begin{eqnarray}
&&\tau\varepsilon\frac{d}{dt}\int\Delta_{-1}\mathbf{u}\cdot\Delta_{-1}\mathbf{E}+\tau\varepsilon\|\Delta_{-1}\mathbf{E}\|^2_{L^2}
+\tau\varepsilon\bar{\psi}(P_{0}\gamma)^{-\frac{1}{2}}\bar{n}^{\frac{3-\gamma}{2}}\|\Delta_{-1}\varrho\|^2_{L^2}\nonumber\\
&=&\bar{n}\tau\varepsilon\|\Delta_{-1}\mathbf{u}\|^2_{L^2}+\tau\int\Delta_{-1}\mathbf{u}\cdot\Delta_{-1}(\nabla\times\mathbf{F})+\tau\varepsilon\int\Delta_{-1}\mathbf{u}\cdot\Delta_{q}(h(\varrho)\mathbf{u})
\nonumber\\&&+\tau\varepsilon\bar{\psi}\int\Delta_{-1}(\tilde{h}(\varrho)\varrho)\Delta_{-1}\varrho-\varepsilon\int\Delta_{-1}\mathbf{u}\cdot\Delta_{-1}\mathbf{E}
\nonumber\\&&-\tau\varepsilon^2\int(\Delta_{-1}\mathbf{u}\times
\overline{\mathbf{B}})\cdot\Delta_{-1}\mathbf{E}-\tau\varepsilon\int\Delta_{-1}(\mathbf{u}\cdot\nabla\mathbf{u})\cdot\Delta_{-1}\mathbf{E}
\nonumber\\&&-\frac{\gamma-1}{2}\tau\varepsilon\int\Delta_{-1}(\varrho\nabla\varrho)\cdot\Delta_{-1}\mathbf{E}-\tau\varepsilon^2\int\Delta_{-1}(\mathbf{u}\times\mathbf{F})\cdot\Delta_{-1}\mathbf{E},\label{R-E67}
\end{eqnarray}
where we have used the third equation of (\ref{R-E17}). From
Cauchy-Schwartz and Young's inequalities, we arrive at
\begin{eqnarray}
&&\tau\varepsilon\frac{d}{dt}\int\Delta_{-1}\mathbf{u}\cdot\Delta_{-1}\mathbf{E}+\frac{\tau\varepsilon}{4}\|\Delta_{-1}\mathbf{E}\|^2_{L^2}\nonumber\\
&\leq&C\Big(\frac{1}{\tau}\|\Delta_{-1}\mathbf{u}\|^2_{L^2}+\tau\|\Delta_{-1}\mathbf{u}\|_{L^2}\|\Delta_{-1}(\nabla\times\mathbf{F})\|_{L^2}\Big)
\nonumber\\&&+C\tau\varepsilon\Big(\|\Delta_{-1}(h(\varrho)\mathbf{u})\|_{L^2}\|\Delta_{-1}\mathbf{u}\|_{L^2}
+\|\Delta_{-1}(\tilde{h}(\varrho)\varrho)\|_{L^2}\|\Delta_{-1}\varrho\|_{L^2}\nonumber\\&&+\|\Delta_{-1}(\mathbf{u}\cdot\nabla\mathbf{u})\|_{L^2}\|\Delta_{-1}\mathbf{E}\|_{L^2}
+\|\Delta_{-1}(\varrho\nabla\varrho)\|_{L^2}\|\Delta_{-1}\mathbf{E}\|_{L^2}\nonumber\\&&+\varepsilon\|\Delta_{-1}(\mathbf{u}\times\mathbf{F})\|_{L^2}\|\Delta_{-1}\mathbf{E}\|_{L^2}\Big).\label{R-E68}
\end{eqnarray}
Then integrating (\ref{R-E68}) in $t\in [0,T]$ gives
\begin{eqnarray}
\hspace{4mm}&&\tau\varepsilon\|\Delta_{-1}\mathbf{E}\|^2_{L_{t}^2(L^2)}\nonumber\\
\hspace{4mm}&\leq&C\tau\varepsilon(\|\Delta_{-1}\mathbf{u}\|_{L_{T}^\infty(L^2)}\|\Delta_{-1}\mathbf{E}\|_{L_{T}^\infty(L^2)}+\|\Delta_{-1}\mathbf{u}_{0}\|_{L^2}\|\Delta_{-1}\mathbf{E}_{0}\|_{L^2})
\nonumber
\\&&\hspace{4mm}+C\Big(\frac{1}{\tau}\|\Delta_{-1}\mathbf{u}\|^2_{L_{T}^2(L^2)}+\tau\|\Delta_{-1}\mathbf{u}\|_{L_{T}^2(L^2)}\|\Delta_{-1}(\nabla\times\mathbf{F})\|_{L_{T}^2(L^2)}\Big)
\nonumber\\&&\hspace{4mm}+C\tau\varepsilon\Big(\|\Delta_{-1}(h(\varrho)\mathbf{u})\|_{L_{T}^2(L^2)}\|\Delta_{-1}\mathbf{u}\|_{L_{T}^2(L^2)}
+\|\Delta_{-1}(\tilde{h}(\varrho)\varrho)\|_{L_{T}^2(L^2)}\nonumber\\&&\hspace{5mm}\cdot\|\Delta_{-1}\varrho\|_{L_{T}^2(L^2)}
+\|\Delta_{-1}(\mathbf{u}\cdot\nabla\mathbf{u})\|_{L_{T}^2(L^2)}\|\Delta_{-1}\mathbf{E}\|_{L_{T}^2(L^2)}
\nonumber\\&&\hspace{4mm}+\|\Delta_{-1}(\varrho\nabla\varrho)\|_{L_{T}^2(L^2)}\|\Delta_{-1}\mathbf{E}\|_{L_{T}^2(L^2)}
+\varepsilon\|\Delta_{-1}(\mathbf{u}\times\mathbf{F})\|_{L_{T}^2(L^2)}\|\Delta_{-1}\mathbf{E}\|_{L_{T}^2(L^2)}\Big).\label{R-E69}
\end{eqnarray}
Therefore, by combining with the high-frequency estimate
(\ref{R-E64}) and low-frequency estimate (\ref{R-E69}), we infer
that for $q\geq-1$,
\begin{eqnarray}
&&\tau\varepsilon2^{2q}\|\Delta_{q}\mathbf{E}\|^2_{L_{t}^2(L^2)}\nonumber\\&\leq&
C2^{2q}(\|(\Delta_{q}\mathbf{u},\Delta_{q}\mathbf{E},\Delta_{q}\mathbf{F})\|^2_{L_{T}^\infty(L^2)}+\|(\Delta_{q}\mathbf{u}_{0},\Delta_{q}\mathbf{E}_{0},\Delta_{q}\mathbf{F}_{0})\|^2_{L^2})
\nonumber
\\&&+C\Big(\frac{1}{\tau}\|\Delta_{q}\mathbf{u}\|^2_{L_{T}^2(L^2)}+\tau\|\Delta_{q}\varrho\|^2_{L_{T}^2(L^2)}
+\tau\|\Delta_{q}\mathbf{u}\|_{L_{T}^2(L^2)}\|\Delta_{q}(\nabla\times\mathbf{F})\|_{L_{T}^2(L^2)}\Big)
\nonumber
\\&&+C\tau\varepsilon\Big\{\|\Delta_{q}(h(\varrho)\mathbf{u})\|_{L_{T}^2(L^2)}\|\Delta_{q}\mathbf{u}\|_{L_{T}^2(L^2)}
+\|\Delta_{q}(\tilde{h}(\varrho)\varrho)\|_{L_{T}^2(L^2)}\|\Delta_{q}\varrho\|_{L_{T}^2(L^2)}
\nonumber\\&&+\Big(\|\Delta_{q}(\mathbf{u}\cdot\nabla\mathbf{u})\|_{L_{T}^2(L^2)}
+\|\Delta_{q}(\varrho\nabla\varrho)\|_{L_{T}^2(L^2)}
+2^{q}\|\Delta_{q}(\widetilde{h}(\varrho)\varrho)\|_{L_{T}^2(L^2)}
\nonumber
\\&&+\varepsilon\|\Delta_{q}(\mathbf{u}\times\mathbf{F})\|_{L_{T}^2(L^2)}\Big)\|\Delta_{q}\mathbf{E}\|_{L_{T}^2(L^2)}
\nonumber
\\&&+\tau\varepsilon^2\|\Delta_{q}(h(\varrho)\mathbf{u})\|_{L_{T}^2(L^2)}\|\Delta_{q}(\nabla\times\mathbf{F})\|_{L_{T}^2(L^2)}\Big\}.\label{R-E70}
\end{eqnarray}
By multiplying the factor $2^{2q(\sigma-1)}$ on both sides of
(\ref{R-E70}), we gather
\begin{eqnarray}
&&\tau\varepsilon2^{2q\sigma}\|\Delta_{q}\mathbf{E}\|^2_{L^2_{t}(L^2)}\nonumber\\
&\leq&Cc_{q}^2(\|(\mathbf{u},\mathbf{E},\mathbf{F})\|^2_{\widetilde{L}^{\infty}_{T}(B^{\sigma}_{2,1})}
+\|(\mathbf{u}_{0},\mathbf{E}_{0},\mathbf{F}_{0})\|^2_{B^{\sigma}_{2,1}})
\nonumber\\&&\hspace{5mm}+C\Big\{\frac{c_{q}^2}{\tau}\|\mathbf{u}\|^2_{\widetilde{L}^2_{T}(B^{\sigma}_{2,1})}+\tau
c_{q}^2\|\varrho\|^2_{\widetilde{L}^2_{T}(B^{\sigma}_{2,1})} +\tau
c_{q}^2\|\mathbf{u}\|_{\widetilde{L}^2_{T}(B^{\sigma}_{2,1})}\|\nabla\times\mathbf{F}\|_{\widetilde{L}^2_{T}(B^{\sigma-1}_{2,1})}
\nonumber\\&&\hspace{5mm}+\tau\varepsilon
c_{q}^2\|h(\varrho)\mathbf{u}\|_{\widetilde{L}^2_{T}(B^{\sigma}_{2,1})}\|\mathbf{u}\|_{\widetilde{L}^2_{T}(B^{\sigma}_{2,1})}
+\tau\varepsilon
c_{q}^2\|\tilde{h}(\varrho)\varrho\|_{\widetilde{L}^2_{T}(B^{\sigma}_{2,1})}\|\varrho\|_{\widetilde{L}^2_{T}(B^{\sigma}_{2,1})}\nonumber
\\&&\hspace{5mm}+\tau\varepsilon
c_{q}^2\Big(\|\mathbf{u}\cdot\nabla\mathbf{u}\|_{\widetilde{L}^2_{T}(B^{\sigma-1}_{2,1})}+\|\varrho\nabla\varrho\|_{\widetilde{L}^2_{T}(B^{\sigma-1}_{2,1})}
+\|\widetilde{h}(\varrho)\varrho\|_{\widetilde{L}^2_{T}(B^{\sigma}_{2,1})}\nonumber\\&&\hspace{5mm}+\varepsilon\|\mathbf{u}\times\mathbf{F}\|_{\widetilde{L}^2_{T}(B^{\sigma-1}_{2,1})}\Big)
\|\mathbf{E}\|_{\widetilde{L}^2_{T}(B^{\sigma}_{2,1})}+\tau\varepsilon^2c^2_{q}\|h(\varrho)\mathbf{u}\|_{\widetilde{L}^2_{T}(B^{\sigma}_{2,1})}\|\nabla\times\mathbf{F}\|_{\widetilde{L}^2_{T}(B^{\sigma-1}_{2,1})}\Big\},\label{R-E71}
\end{eqnarray}
where $\{c_{q}\}$ denotes some sequence which satisfies
$\|(c_{q})\|_{ {l^{1}}}\leq 1$.

Then it follows from Young's inequality that
\begin{eqnarray}
&&\sqrt{\tau\varepsilon}2^{q\sigma}\|\Delta_{q}\mathbf{E}\|_{L^2_{T}(L^2)}\nonumber\\
&\leq&Cc_{q}(\|(\mathbf{u},\mathbf{E},\mathbf{F})\|_{\widetilde{L}^{\infty}_{T}(B^{\sigma}_{2,1})}
+\|(\mathbf{u}_{0},\mathbf{E}_{0},\mathbf{F}_{0})\|_{B^{\sigma}_{2,1}})
\nonumber\\&&+C\Big\{\frac{c_{q}}{\sqrt{\tau}}\|\mathbf{u}\|_{\widetilde{L}^2_{T}(B^{\sigma}_{2,1})}+\sqrt{\tau}c_{q}\|\varrho\|_{\widetilde{L}^2_{T}(B^{\sigma}_{2,1})}
+\frac{c_{q}}{\sqrt{\varepsilon}}\|\nabla\times\mathbf{F}\|_{\widetilde{L}^2_{T}(B^{\sigma-1}_{2,1})}
\nonumber\\&&+c_{q}\sqrt{\|\mathbf{u}\|_{\widetilde{L}^\infty_{T}(B^{\sigma}_{2,1})}}\Big(\sqrt{\tau}\|\varrho\|_{\widetilde{L}^2_{T}(B^{\sigma}_{2,1})}+\frac{1}{\sqrt{\tau}}\|\mathbf{u}\|_{\widetilde{L}^2_{T}(B^{\sigma}_{2,1})}\Big)
\nonumber\\&&+c_{q}\sqrt{\|\mathbf{u}\|_{\widetilde{L}^\infty_{T}(B^{\sigma}_{2,1})}}\Big(\frac{1}{\sqrt{\tau}}\|\mathbf{u}\|_{\widetilde{L}^2_{T}(B^{\sigma}_{2,1})}+\sqrt{\tau\varepsilon}\|\mathbf{E}\|_{\widetilde{L}^2_{T}(B^{\sigma}_{2,1})}\Big)
\nonumber\\&&+c_{q}\sqrt{\|\varrho\|_{\widetilde{L}^\infty_{T}(B^{\sigma}_{2,1})}}
(\sqrt{\tau}\|\varrho\|_{\widetilde{L}^2_{T}(B^{\sigma}_{2,1})}+\sqrt{\tau\varepsilon}\|\mathbf{E}\|_{\widetilde{L}^2_{T}(B^{\sigma}_{2,1})})
\nonumber\\&&+c_{q}\sqrt{\|\mathbf{F}\|_{\widetilde{L}^\infty_{T}(B^{\sigma}_{2,1})}}
\Big(\frac{1}{\sqrt{\tau}}\|\mathbf{u}\|_{\widetilde{L}^2_{T}(B^{\sigma}_{2,1})}+\sqrt{\tau\varepsilon}\|\mathbf{E}\|_{\widetilde{L}^2_{T}(B^{\sigma}_{2,1})}\Big)\Big\}
\nonumber\\&&+c_{q}\sqrt{\|\varrho\|_{\widetilde{L}^\infty_{T}(B^{\sigma}_{2,1})}}\Big(\frac{1}{\sqrt{\tau}}\|\mathbf{u}\|_{\widetilde{L}^2_{T}(B^{\sigma}_{2,1})}
+\frac{1}{\sqrt{\varepsilon}}\|\nabla\times\mathbf{F}\|_{\widetilde{L}^2_{T}(B^{\sigma-1}_{2,1})}\Big),\label{R-E72}
\end{eqnarray}
where we has used Proposition \ref{prop2.2}. Finally, we sum up
(\ref{R-E72}) on $q\geq-1$ and deduce the inequality (\ref{R-E58})
immediately.
\end{proof}\\

\begin{lemma}\label{lem4.5}
If $W\in \widetilde{\mathcal{C}}_{T}(B^{\sigma}_{2,1})\cap
\widetilde{\mathcal{C}}^1_{T}(B^{\sigma-1}_{2,1})$ is a solution of
(\ref{R-E17})-(\ref{R-E18}) for any $T>0$ and
$0<\tau,\varepsilon\leq1$, then the following estimate holds:
\begin{eqnarray}
&&\frac{1}{\sqrt{\varepsilon}}\|\nabla\mathbf{F}\|_{\widetilde{L}^2_{T}(B^{\sigma-1}_{2,1})}
\nonumber\\&\leq&C(\|(\mathbf{E},\mathbf{F})\|_{\widetilde{L}^\infty_{T}(B^{\sigma}_{2,1})}
+\|(\mathbf{E}_{0},\mathbf{F}_{0})\|_{B^{\sigma}_{2,1}})\nonumber\\&&
+C\Big\{\Big\|\frac{\mathbf{u}}{\sqrt{\tau}}\Big\|_{\widetilde{L}^2_{T}(B^{\sigma}_{2,1})}
+\sqrt{\|\varrho\|_{\widetilde{L}^\infty_{T}(B^{\sigma}_{2,1})}}\Big(\Big\|\frac{\mathbf{u}}{\sqrt{\tau}}\Big\|_{\widetilde{L}^2_{T}(B^{\sigma}_{2,1})}
+\Big\|\frac{\nabla\mathbf{F}}{\sqrt{\varepsilon}}\Big\|_{\widetilde{L}^2_{T}(B^{\sigma-1}_{2,1})}\Big)\Big\},\label{R-E73}
\end{eqnarray}
where $C>0$ is a uniform positive constant independent of
$(\tau,\varepsilon)$.
\end{lemma}

\begin{proof} Multiply the third
equation of (\ref{R-E17}) by
$-\Delta_{q}(\nabla\times\mathbf{F})(q\geq-1)$ and integrate the
resulting quality over $\mathbf{R}^{N}$. Then integration by parts
implies
\begin{eqnarray}
&&\frac{d}{dt}\int\Delta_{q}(\nabla\times\mathbf{E})\cdot\Delta_{q}\mathbf{F}+\frac{1}{\varepsilon}\|\Delta_{q}(\nabla\times\mathbf{F})\|^2_{L^2}
\nonumber\\&=&\int\Delta_{q}(\nabla\times\mathbf{E})\cdot\Delta_{q}\mathbf{F}_{t}-\bar{n}\int\Delta_{q}\mathbf{u}\cdot\Delta_{q}(\nabla\times\mathbf{F})
\nonumber\\&&-\int\Delta_{q}(h(\varrho)\mathbf{u})\cdot\Delta_{q}(\nabla\times\mathbf{F}).\label{R-E74}
\end{eqnarray}
Substituting the fourth equation of (\ref{R-E17}) into the first
term of (\ref{R-E74}), by Cauchy-Schwartz inequality, leads to
\begin{eqnarray}
&&\frac{d}{dt}\int\Delta_{q}(\nabla\times\mathbf{E})\cdot\Delta_{q}\mathbf{F}+\frac{1}{\varepsilon}\|\Delta_{q}(\nabla\times\mathbf{F})\|^2_{L^2}
+\frac{1}{\varepsilon}\|\Delta_{q}(\nabla\times\mathbf{E})\|^2_{L^2}\nonumber\\&\leq&\bar{n}\|\Delta_{q}\mathbf{u}\|_{L^2}\|\Delta_{q}(\nabla\times\mathbf{F})\|_{L^2}
+\|\Delta_{q}(h(\varrho)\mathbf{u})\|_{L^2}\|\Delta_{q}(\nabla\times\mathbf{F})\|_{L^2}.\label{R-E75}
\end{eqnarray}
Due to the incompressible condition of $\mathbf{F}$ in
(\ref{R-E17}), by integrating (\ref{R-E75}) with respect to
$t\in[0,T]$, we easily derive

\begin{eqnarray}
&&\frac{1}{\varepsilon}\|\Delta_{q}\nabla\mathbf{F}\|^2_{L^2_{t}(L^2)}
\nonumber\\&\leq&\|\Delta_{q}(\nabla\times\mathbf{E})\|_{L^\infty_{T}(L^2)}
\|\Delta_{q}\mathbf{F}\|_{L^\infty_{T}(L^2)}+\|\Delta_{q}(\nabla\times\mathbf{E}_{0})\|_{L^2}
\|\Delta_{q}\mathbf{F}_{0}\|_{L^2}\nonumber\\&&
+\bar{n}\|\Delta_{q}\mathbf{u}\|_{L^2_{T}(L^2)}\|\Delta_{q}(\nabla\times\mathbf{F})\|_{L^2_{T}(L^2)}
\nonumber\\&&+\|\Delta_{q}(h(\varrho)\mathbf{u})\|_{L^2_{T}(L^2)}\|\Delta_{q}(\nabla\times\mathbf{F})\|_{L^2_{T}(L^2)}.\label{R-E76}
\end{eqnarray}
Noticing that the regularity of $\mathbf{E}$ in the assumption of
Lemma \ref{lem4.5}, we multiply (\ref{R-E76}) by the factor
$2^{2q(\sigma-1)}$ to get
\begin{eqnarray}
&&\frac{1}{\varepsilon}2^{2q(\sigma-1)}\|\Delta_{q}\nabla\mathbf{F}\|^2_{L^2_{t}(L^2)}
\nonumber\\&\leq&C\Big\{c^2_{q}\|\mathbf{E}\|_{\widetilde{L}^\infty_{T}(B^{\sigma}_{2,1})}
\|\mathbf{F}\|_{\widetilde{L}^\infty_{T}(B^{\sigma-1}_{2,1})}+c^2_{q}\|\mathbf{E}_{0}\|_{B^{\sigma}_{2,1}}
\|\mathbf{F}_{0}\|_{B^{\sigma-1}_{2,1}}\nonumber\\&&
+c^2_{q}\|\mathbf{u}\|_{\widetilde{L}^2_{T}(B^{\sigma-1}_{2,1})}\|\nabla\times\mathbf{F}\|_{\widetilde{L}^2_{T}(B^{\sigma-1}_{2,1})}
\nonumber\\&&+c^2_{q}\|h(\varrho)\mathbf{u}\|_{\widetilde{L}^2_{T}(B^{\sigma-1}_{2,1})}\|\nabla\times\mathbf{F}\|_{\widetilde{L}^2_{T}(B^{\sigma-1}_{2,1})}\Big\},\label{R-E77}
\end{eqnarray}
where $\{c_{q}\}$ denotes some sequence which satisfies
$\|(c_{q})\|_{ {l^{1}}}\leq 1$.

Furthermore, we apply Young's inequality to (\ref{R-E77}) and obtain
\begin{eqnarray}
&&\frac{1}{\sqrt{\varepsilon}}2^{q(\sigma-1)}\|\Delta_{q}\nabla\mathbf{F}\|_{L^2_{T}(L^2)}
\nonumber\\&\leq&Cc_{q}\Big\{\|(\mathbf{E},\mathbf{F})\|_{\widetilde{L}^\infty_{T}(B^{\sigma}_{2,1})}
+\|(\mathbf{E}_{0},\mathbf{F}_{0})\|_{B^{\sigma}_{2,1}}\Big\}\nonumber\\&&
+C\Big\{c_{q}\sqrt{\|\mathbf{u}\|_{\widetilde{L}^2_{T}(B^{\sigma}_{2,1})}\|\nabla\times\mathbf{F}\|_{\widetilde{L}^2_{T}(B^{\sigma-1}_{2,1})}}
\nonumber\\&&+c_{q}\sqrt{\|h(\varrho)\mathbf{u}\|_{\widetilde{L}^2_{T}(B^{\sigma}_{2,1})}\|\nabla\times\mathbf{F}\|_{\widetilde{L}^2_{T}(B^{\sigma-1}_{2,1})}}\Big\}.\label{R-E78}
\end{eqnarray}
Finally, after summing up (\ref{R-E78}) on $q\geq-1$, it follows
from Proposition \ref{prop2.2} that
\begin{eqnarray*}
&&\frac{1}{\sqrt{\varepsilon}}\|\nabla\mathbf{F}\|_{\widetilde{L}^2_{T}(B^{\sigma-1}_{2,1})}
\nonumber\\&\leq&C(\|(\mathbf{E},\mathbf{F})\|_{\widetilde{L}^\infty_{T}(B^{\sigma}_{2,1})}
+\|(\mathbf{E}_{0},\mathbf{F}_{0})\|_{B^{\sigma}_{2,1}})
+C\Big\{\frac{1}{\sqrt{\tau}}\|\mathbf{u}\|_{\widetilde{L}^2_{T}(B^{\sigma}_{2,1})}
\nonumber\\&&+\sqrt{\|\varrho\|_{\widetilde{L}^\infty_{T}(B^{\sigma}_{2,1})}}\Big(\frac{1}{\sqrt{\tau}}\|\mathbf{u}\|_{\widetilde{L}^2_{T}(B^{\sigma}_{2,1})}
+\frac{1}{\sqrt{\varepsilon}}\|\nabla\times\mathbf{F}\|_{\widetilde{L}^2_{T}(B^{\sigma-1}_{2,1})}\Big)\Big\}.
\end{eqnarray*}
This is just the inequality (\ref{R-E73}). Hence, the proof of Lemma
\ref{lem4.5} is complete.
\end{proof}

\begin{remark}
\rm In the proof of Lemma \ref{lem4.5}, the dissipation rate of
$\mathbf{F}$ is not available  due to the absence of low-frequency
estimate on $\|\Delta_{-1}\mathbf{F}\|_{L^2_{T}(L^2)}$. This is a
key reason that Chemin-Lerner's spaces with critical regularity are
first introduced to establish the global existence of uniform
classical solutions. Otherwise, we need to add a little regularity
in order to ensure that the Besov spaces (in $x$) are still
continuously embedded in $\mathcal{C}^1(\textbf{R}^N)$ spaces. For
the similar details, the reader is referred to \cite{FX}.

\end{remark}

Having these lemmas proved above, the proof of Proposition
\ref{prop4.1}
can be finished.\\

\noindent\textit{\underline{Proof of Proposition \ref{prop4.1}.}}
Combing (\ref{R-E32}), (\ref{R-E36}), (\ref{R-E58}) and
(\ref{R-E73}), we end up with
\begin{eqnarray}
&&\|W\|_{\widetilde{L}_{T}^{\infty}(B^{\sigma}_{2,1})}+K_{1}\sqrt{\tau}\|\varrho\|_{\widetilde{L}^2_{T}(B^{\sigma}_{2,1})}+\sqrt{\frac{\mu_{2}}{\tau}}\|\textbf{u}\|_{\widetilde{L}^2_{T}(B^{\sigma}_{2,1})}
\nonumber
\\&&+K_{2}\sqrt{\tau\varepsilon}\|\mathbf{E}\|_{\widetilde{L}^2_{T}(B^{\sigma}_{2,1})}+\frac{K_{3}}{\sqrt{\varepsilon}}\|\nabla\mathbf{F}\|_{\widetilde{L}^2_{T}(B^{\sigma-1}_{2,1})}
\nonumber
\\&\leq&C\|W_{0}\|_{B^{\sigma}_{2,1}}+C\sqrt{\|W\|_{\widetilde{L}_{T}^{\infty}(B^{\sigma}_{2,1})}}
\Big\|\Big(\sqrt{\tau}\varrho,\frac{1}{\sqrt{\tau}}\textbf{u}\Big)\Big\|_{\widetilde{L}_{T}^2(B^{\sigma}_{2,1})}
\nonumber\\
&&+CK_{1}\Big\{(\|W_{I}\|_{\widetilde{L}^\infty_{T}(B^{\sigma}_{2,1})}+\|W_{I}(0)\|_{B^{\sigma}_{2,1}})
+\frac{1}{\sqrt{\tau}}\|\textbf{u}\|_{\widetilde{L}^2_{T}(B^{\sigma}_{2,1})}
\nonumber\\
&&+\sqrt{\|(\varrho,\textbf{u},\mathbf{F})\|_{\widetilde{L}^\infty_{T}(B^{\sigma}_{2,1})}}\Big\|\Big(\sqrt{\tau}\varrho,\frac{1}{\sqrt{\tau}}\textbf{u}\Big)\Big\|_{\widetilde{L}_{T}^2(B^{\sigma}_{2,1})}\Big\}
\nonumber\\
&&+CK_{2}\Big\{(\|(\mathbf{u},\mathbf{E},\mathbf{F})\|_{\widetilde{L}^{\infty}_{T}(B^{\sigma}_{2,1})}+\|(\mathbf{u}_{0},\mathbf{E}_{0},\mathbf{F}_{0})\|_{B^{\sigma}_{2,1}})
\nonumber\\&&+\Big\|\Big(\sqrt{\tau}\varrho,\frac{1}{\sqrt{\tau}}\textbf{u}\Big)\Big\|_{\widetilde{L}_{T}^2(B^{\sigma}_{2,1})}
+\frac{1}{\sqrt{\varepsilon}}\|\nabla\mathbf{F}\|_{\widetilde{L}^2_{T}(B^{\sigma-1}_{2,1})}
\nonumber\\&&+\sqrt{\|(\varrho,\mathbf{u},\mathbf{F})\|_{\widetilde{L}^\infty_{T}(B^{\sigma}_{2,1})}}\Big[\Big\|\Big(\sqrt{\tau}\varrho,\frac{1}{\sqrt{\tau}}\textbf{u},\sqrt{\tau\varepsilon}\mathbf{E}\Big)\Big\|_{\widetilde{L}^2_{T}(B^{\sigma}_{2,1})}
\nonumber\\&&+\Big\|\frac{1}{\sqrt{\varepsilon}}\nabla\textbf{F}\Big\|_{\widetilde{L}^2_{T}(B^{\sigma-1}_{2,1})}\Big]
\Big\}
+CK_{3}\Big\{(\|(\mathbf{E},\mathbf{F})\|_{\widetilde{L}^\infty_{T}(B^{\sigma}_{2,1})}
+\|(\mathbf{E}_{0},\mathbf{F}_{0})\|_{B^{\sigma}_{2,1}})\nonumber\\&&
+\frac{1}{\sqrt{\tau}}\|\mathbf{u}\|_{\widetilde{L}^2_{T}(B^{\sigma}_{2,1})}
+\sqrt{\|\varrho\|_{\widetilde{L}^\infty_{T}(B^{\sigma}_{2,1})}}\Big(\frac{1}{\sqrt{\tau}}\|\mathbf{u}\|_{\widetilde{L}^2_{T}(B^{\sigma}_{2,1})}
\nonumber\\&&+\frac{1}{\sqrt{\varepsilon}}\|\nabla\mathbf{F}\|_{\widetilde{L}^2_{T}(B^{\sigma-1}_{2,1})}\Big)\Big\},\label{R-E80}
\end{eqnarray}
where $K_{1},K_{2}$ and $K_{3}$ are some uniform positive constants
(independent of $(\tau,\varepsilon)$) to be determined. In order to
eliminate the terms
$\|(\varrho,\mathbf{u},\mathbf{E},\mathbf{F})\|_{\widetilde{L}^\infty_{T}(B^{\sigma}_{2,1})}$,
$\|\sqrt{\tau}\varrho\|_{\widetilde{L}^2_{T}(B^{\sigma}_{2,1})}$,\
$\|\mathbf{u}/\sqrt{\tau}\|_{\widetilde{L}^2_{T}(B^{\sigma}_{2,1})}$
and
$\|\nabla\mathbf{F}/\sqrt{\varepsilon}\|_{\widetilde{L}^2_{T}(B^{\sigma-1}_{2,1})}$
arising in the right-hand side of (\ref{R-E80}), we may confine the
constants to the following cases:
$$K_{1}\leq\min\Big\{\frac{1}{4C},\frac{\sqrt{\mu_{2}}}{4C}\Big\},\,\, K_{2}\leq\min\Big\{\frac{1}{4C},\frac{K_{1}}{2C},\frac{\sqrt{\mu_{2}}}{4C},\frac{K_{3}}{2C}\Big\}
,\,\, K_{3}\leq\min\Big\{\frac{1}{4C},
\frac{\sqrt{\mu_{2}}}{4C}\Big\}. $$ Furthermore, it is not difficult
to obtain
\begin{eqnarray}
&&\frac{1}{2}\|W\|_{\widetilde{L}_{T}^{\infty}(B^{\sigma}_{2,1})}+\frac{K_{1}\sqrt{\tau}}{2}\|\varrho\|_{\widetilde{L}^2_{T}(B^{\sigma}_{2,1})}+\frac{\sqrt{\mu_{2}}}{4\sqrt{\tau}}\|\textbf{u}\|_{\widetilde{L}^2_{T}(B^{\sigma}_{2,1})}
\nonumber
\\&&+K_{2}\sqrt{\tau\varepsilon}\|\mathbf{E}\|_{\widetilde{L}^2_{T}(B^{\sigma}_{2,1})}+\frac{K_{3}}{2\sqrt{\varepsilon}}\|\nabla\mathbf{F}\|_{\widetilde{L}^2_{T}(B^{\sigma-1}_{2,1})}
\nonumber
\\&\leq&C\|W_{0}\|_{B^{\sigma}_{2,1}}+C\sqrt{\|W\|_{\widetilde{L}_{T}^{\infty}(B^{\sigma}_{2,1})}}
\Big\|\Big(\sqrt{\tau}\varrho,\frac{1}{\sqrt{\tau}}\textbf{u}\Big)\Big\|_{\widetilde{L}_{T}^2(B^{\sigma}_{2,1})}\nonumber
\\
&&+CK_{1}\Big\{\|(\varrho_{0},\textbf{u}_{0})\|_{B^{\sigma}_{2,1}}
+\sqrt{\|(\varrho,\textbf{u},\mathbf{F})\|_{\widetilde{L}^\infty_{T}(B^{\sigma}_{2,1})}}\nonumber\\&&\hspace{1mm}\cdot\Big\|\Big(\sqrt{\tau}\varrho,\frac{1}{\sqrt{\tau}}\textbf{u}\Big)\Big\|_{\widetilde{L}_{T}^2(B^{\sigma}_{2,1})}\Big\}
+CK_{2}\Big\{\|(\mathbf{u}_{0},\mathbf{E}_{0},\mathbf{F}_{0})\|_{B^{\sigma}_{2,1}}
\nonumber\\
&&+\sqrt{\|(\varrho,\mathbf{u},\mathbf{F})\|_{\widetilde{L}^\infty_{T}(B^{\sigma}_{2,1})}}\Big[\Big\|\Big(\sqrt{\tau}\varrho,\frac{1}{\sqrt{\tau}}\textbf{u},\sqrt{\tau\varepsilon}\mathbf{E}\Big)\Big\|_{\widetilde{L}^2_{T}(B^{\sigma}_{2,1})}
\nonumber\\
&&+\Big\|\frac{\nabla\textbf{F}}{\sqrt{\varepsilon}}\Big\|_{\widetilde{L}^2_{T}(B^{\sigma-1}_{2,1})}\Big]
\Big\}
+CK_{3}\Big\{\|(\mathbf{E}_{0},\mathbf{F}_{0})\|_{B^{\sigma}_{2,1}}
\nonumber\\&&
+\sqrt{\|\varrho\|_{\widetilde{L}^\infty_{T}(B^{\sigma}_{2,1})}}\Big(\frac{1}{\sqrt{\tau}}\|\mathbf{u}\|_{\widetilde{L}^2_{T}(B^{\sigma}_{2,1})}
+\frac{1}{\sqrt{\varepsilon}}\|\nabla\mathbf{F}\|_{\widetilde{L}^2_{T}(B^{\sigma-1}_{2,1})}\Big)\Big\}
\nonumber\\&\leq&C\|W_{0}\|_{B^{\sigma}_{2,1}}
+C\sqrt{\|W\|_{\widetilde{L}_{T}^{\infty}(B^{\sigma}_{2,1})}}\Big\{
\Big\|\Big(\sqrt{\tau}\varrho,\frac{1}{\sqrt{\tau}}\textbf{u},\sqrt{\tau\varepsilon}\mathbf{E}\Big)\Big\|_{\widetilde{L}_{T}^2(B^{\sigma}_{2,1})}
\nonumber\\&&+\Big\|\frac{\nabla\textbf{F}}{\sqrt{\varepsilon}}\Big\|_{\widetilde{L}^2_{T}(B^{\sigma-1}_{2,1})}\Big\}
\nonumber\\
\hspace{10mm}&\leq&C\|W_{0}\|_{B^{\sigma}_{2,1}}+C\sqrt{\delta_{1}}
\Big\{\Big\|\Big(\sqrt{\tau}\varrho,\frac{1}{\sqrt{\tau}}\textbf{u},\sqrt{\tau\varepsilon}\mathbf{E}\Big)\Big\|_{\widetilde{L}_{T}^2(B^{\sigma}_{2,1})}
+\Big\|\frac{\nabla\textbf{F}}{\sqrt{\varepsilon}}\Big\|_{\widetilde{L}^2_{T}(B^{\sigma-1}_{2,1})}\Big\},\label{R-E81}
\end{eqnarray}
where we used the \textit{a priori} assumption (\ref{R-E30}) in the
last step of (\ref{R-E81}).

Finally, we choose the positive constant $\delta_{1}$ such that
$$C\sqrt{\delta_{1}}<\min\Big\{\frac{K_{1}}{2},\frac{\sqrt{\mu_{2}}}{4},K_{2},\frac{K_{3}}{2}\Big\},$$
then the inequality (\ref{R-E31}) follows immediately.  This
finishes the proof of Proposition \ref{prop4.1} eventually.

\subsection{Non-relativistic limit}\label{sec:3.3}
In this section, we justify the non-relativistic limit of the system
(\ref{R-E1})-(\ref{R-E3}) with $\tau=1$.\\

\noindent\textit{\underline{Proof of Theorem \ref{thm1.3}.}} For any
fixed $T>0$, let
$(n^{\varepsilon},\mathbf{u}^{\varepsilon},\mathbf{E}^{\varepsilon},\mathbf{B}^{\varepsilon})$
be the global solution of (\ref{R-E1})-(\ref{R-E3}) given by Theorem
\ref{thm1.2}. It follows from the uniform energy estimate
(\ref{R-E11}) and Remark \ref{rem2.1} that
\begin{eqnarray}(n^{\varepsilon}-\bar{n}, \textbf{u}^{\varepsilon})\in L^\infty_{T}(B^{\sigma}_{2,1})\cap L^{2}_{T}(B^{\sigma}_{2,1}),\label{R-E82}\end{eqnarray}
\begin{eqnarray}\textbf{E}^{\varepsilon}\in L^\infty_{T}(B^{\sigma}_{2,1}),\ \ \sqrt{\varepsilon}\textbf{E}^{\varepsilon}\in L^{2}_{T}(B^{\sigma}_{2,1}),\label{R-E83}\end{eqnarray}
\begin{eqnarray}\textbf{B}^{\varepsilon}-\overline{\mathbf{B}}\in L^\infty_{T}(B^{\sigma}_{2,1}),\ \ \frac{\nabla\textbf{B}^{\epsilon}}{\sqrt{\varepsilon}}\in L^{2}_{T}(B^{\sigma-1}_{2,1}),\label{R-E84}\end{eqnarray}
uniformly in $\varepsilon$. Note that (\ref{R-E84}), we deduce
\begin{eqnarray}
\Big\{\int^{T}_{0}\|\nabla\textbf{B}^{\varepsilon}(t,\cdot)\|^2_{B^{\sigma-1}_{2,1}}dt\Big\}^{1/2}&=&\sqrt{\varepsilon}\Big\{\int^{T}_{0}\Big\|\frac{\nabla\textbf{B}^{\varepsilon}(t,\cdot)}{\sqrt{\varepsilon}}\Big\|^2_{B^{\sigma-1}_{2,1}}dt\Big\}^{1/2}
\nonumber\\&\leq& C\sqrt{\varepsilon}\rightarrow0, \ \ \ \mbox{as}\
\ \ \varepsilon\rightarrow0.\label{R-E85}
\end{eqnarray}
That is,
\begin{eqnarray}
\{\nabla\textbf{B}^{\varepsilon}\}\rightarrow \textbf{0}\ \ \ \
\mbox{strongly in}\ \ \ L^{2}_{T}(B^{\sigma-1}_{2,1}),\ \ \
\mbox{as}\ \ \ \varepsilon\rightarrow0.\label{R-E86}
\end{eqnarray}
Moreover, with the help of (\ref{R-E1}), we have
\begin{eqnarray}(n^{\varepsilon}_{t}, \textbf{u}^{\varepsilon}_{t})\in L^{2}_{T}(B^{\sigma-1}_{2,1}),\label{R-E87}\end{eqnarray}
\begin{eqnarray}
\sqrt{\varepsilon}\textbf{E}^{\varepsilon}_{t}\in
L^{2}_{T}(B^{\sigma-1}_{2,1}),\label{R-E88}
\end{eqnarray}
uniformly in $\varepsilon$.

According to (\ref{R-E82})-(\ref{R-E84}) and
(\ref{R-E87})-(\ref{R-E88}), it can be derived from Proposition
\ref{prop2.1} and Aubin-Lions compactness lemma in \cite{S} that
there exists some function $(n^{0},\textbf{u}^{0},\textbf{E}^{0})\in
\mathcal{C}([0,\infty), \bar{n}+B^{\sigma}_{2,1})\times
\mathcal{C}([0,\infty), B^{\sigma}_{2,1})\times
\mathcal{C}([0,\infty), B^{\sigma}_{2,1})$ such that the sequences
(up to subsequences) as $\varepsilon\rightarrow0$, it holds that
\begin{eqnarray}
\{n^{\varepsilon}\}\rightarrow n^{0}\ \ \ \ \mbox{strongly in}\ \ \
\mathcal{C}([0,T],(B^{\sigma-\delta}_{2,1})_{\mathrm{loc}}),\label{R-E89}
\end{eqnarray}
\begin{eqnarray}
\{\textbf{u}^{\varepsilon}\}\rightarrow \textbf{u}^{0}\ \ \ \
\mbox{strongly
 in}\ \ \ \mathcal{C}([0,T],(B^{\sigma-\delta}_{2,1})_{\mathrm{loc}}),\label{R-E90}
\end{eqnarray}
\begin{eqnarray}
\{\sqrt{\varepsilon}\textbf{E}^{\varepsilon}\}\rightarrow
\textbf{0}\ \ \ \ \mbox{strongly
 in}\ \ \ \mathcal{C}([0,T],(B^{\sigma-\delta}_{2,1})_{\mathrm{loc}}),\label{R-E91}
\end{eqnarray}
\begin{eqnarray}
\{\textbf{E}^{\varepsilon}\}\rightharpoonup \textbf{E}^{0}\ \ \ \
\mbox{weakly$^{\star}$
 in}\ \ \ L^\infty_{T}(B^{\sigma}_{2,1}),\label{R-E92}
\end{eqnarray}
\begin{eqnarray}
\{\textbf{B}^{\varepsilon}\}\rightharpoonup \overline{\mathbf{B}}\ \
\ \ \mbox{weakly$^{\star}$
 in}\ \ \ L^\infty_{T}(B^{\sigma}_{2,1}),\label{R-E93}
\end{eqnarray}
for any $T>0$ and $\delta\in(0,1)$. Thus, in the system
(\ref{R-E1})-(\ref{R-E3}), the uniform bounded properties
(\ref{R-E82})-(\ref{R-E84}) as well as the convergence properties
(\ref{R-E85}) and (\ref{R-E89})-(\ref{R-E93}) allow us to pass to
the limit $\varepsilon\rightarrow0$ in the sense of distributions,
which implies that $(n^{0},\textbf{u}^{0},\textbf{E}^{0})$ is a
global weak solution to the Euler-Poisson equations (\ref{R-E6})
satisfying (\ref{R-E12}). This completes the proof of Theorem
\ref{thm1.3}.

\subsection{Relaxation limit}\label{sec:3.4}
In this section, we prove the relaxation limit of
(\ref{R-E1})-(\ref{R-E3}) with
$\varepsilon=1$.  \\

\noindent\textit{\underline{Proof of Theorem \ref{thm1.4}.}} From
the scaled variable transform (\ref{R-E7}) and the uniform energy
estimate (\ref{R-E11}) in Theorem \ref{thm1.2}, it is shown that
$(n^{\tau},\textbf{u}^{\tau},\textbf{E}^{\tau},\textbf{B}^{\tau})$
is a unique global solution of the system (\ref{R-E8}) and
(\ref{R-E13}), furthermore, for any fixed $T>0$, we have
\begin{eqnarray}n^{\tau}-\bar{n}\in L^\infty_{T}(B^{\sigma}_{2,1})\cap L^{2}_{T}(B^{\sigma}_{2,1}),\label{R-E94}\end{eqnarray}
\begin{eqnarray}
\tau\textbf{u}^{\tau}\in L^\infty_{T}(B^{\sigma}_{2,1}),\ \ \
\textbf{u}^{\tau}\in L^2_{T}(B^{\sigma}_{2,1}),\label{R-E95}
\end{eqnarray}
\begin{eqnarray}\textbf{E}^{\tau}\in L^\infty_{T}(B^{\sigma}_{2,1})\cap L^{2}_{T}(B^{\sigma}_{2,1}),\label{R-E96}\end{eqnarray}
\begin{eqnarray}\textbf{B}^{\tau}-\overline{\mathbf{B}}\in L^\infty_{T}(B^{\sigma}_{2,1}),\ \ \ \frac{\nabla\textbf{B}^{\tau}}{\sqrt{\tau}}\in L^{2}_{T}(B^{\sigma-1}_{2,1}),\label{R-E97}\end{eqnarray}
uniformly in $\tau$. Similar to (\ref{R-E85}), it follows from
(\ref{R-E97}) that
\begin{eqnarray}
\{\nabla\textbf{B}^{\tau}\}\rightarrow \textbf{0}\ \ \ \
\mbox{strongly in}\ \ \ L^{2}_{T}(B^{\sigma-1}_{2,1}),\ \ \
\mbox{as}\ \ \ \tau\rightarrow0.\label{R-E98}
\end{eqnarray}
Moreover, from the equations (\ref{R-E8}), we conclude that
\begin{eqnarray}n^{\tau}_{t}\in L^{2}_{T}(B^{\sigma-1}_{2,1}),\label{R-E99}\end{eqnarray}
\begin{eqnarray}
\tau^2\textbf{u}^{\tau}_{t}\in
L^{2}_{T}(B^{\sigma-1}_{2,1}),\label{R-E100}
\end{eqnarray}
\begin{eqnarray}
\sqrt{\tau}\textbf{E}^{\tau}_{t}\in
L^{2}_{T}(B^{\sigma-1}_{2,1}),\label{R-E101}
\end{eqnarray}
uniformly in $\tau$.

Together with (\ref{R-E94})-(\ref{R-E97}) and
(\ref{R-E99})-(\ref{R-E101}), it follows from Proposition
\ref{prop2.1} and Aubin-Lions compactness lemma in \cite{S} that
there exists some function $(\mathcal{N},\mathcal{U},\mathcal{E})\in
\mathcal{C}([0,\infty), \bar{n}+B^{\sigma}_{2,1})\times
L^2([0,\infty), B^{\sigma}_{2,1})\times \mathcal{C}([0,\infty),
B^{\sigma}_{2,1})$ such that the sequences (up to subsequences) as
$\tau\rightarrow0$, it holds that
\begin{eqnarray}
\{n^{\tau}\}\rightarrow \mathcal{N}\ \ \ \ \mbox{strongly in}\ \ \
\mathcal{C}([0,T],(B^{\sigma-\delta}_{2,1})_{\mathrm{loc}}),\label{R-E102}
\end{eqnarray}
\begin{eqnarray}
\{\tau^2\textbf{u}^{\tau}\}\rightarrow \textbf{0}\ \ \ \
\mbox{strongly
 in}\ \ \ \mathcal{C}([0,T],(B^{\sigma-\delta}_{2,1})_{\mathrm{loc}}),\label{R-E103}
\end{eqnarray}
\begin{eqnarray}
\{\textbf{u}^{\tau}\}\rightharpoonup \mathcal{U}\ \ \ \ \mbox{weakly
 in}\ \ \ L^2_{T}(B^{\sigma}_{2,1}),\label{R-E104}
\end{eqnarray}
\begin{eqnarray}
\{\sqrt{\tau}\textbf{E}^{\tau}\}\rightarrow \textbf{0}\ \ \ \
\mbox{strongly
 in}\ \ \ \mathcal{C}([0,T],(B^{\sigma-\delta}_{2,1})_{\mathrm{loc}}),\label{R-E105}
\end{eqnarray}
\begin{eqnarray}
\{\textbf{E}^{\tau}\}\rightharpoonup \mathcal{E}\ \ \ \
\mbox{weakly$^{\star}$
 in}\ \ \ L^\infty_{T}(B^{\sigma}_{2,1}),\label{R-E106}
\end{eqnarray}
\begin{eqnarray}
\{\textbf{B}^{\tau}\}\rightharpoonup \overline{\mathbf{B}}\ \ \ \
\mbox{weakly$^{\star}$
 in}\ \ \ L^\infty_{T}(B^{\sigma}_{2,1}),\label{R-E107}
\end{eqnarray}
for any $T>0$ and $\delta\in(0,1)$. Thus, the uniform bounded
properties (\ref{R-E94})-(\ref{R-E97}) as well as the convergence
properties (\ref{R-E98}) and (\ref{R-E102})-(\ref{R-E107}) allow us
to pass to the limit $\tau\rightarrow0$ in the system (\ref{R-E8})
and (\ref{R-E13}) in the sense of distributions, which implies that
$(\mathcal{N},\mathcal{E})$ is a global weak solution to the
drift-diffusion equations (\ref{R-E9}) satisfying (\ref{R-E14}).
Hence, the proof of Theorem \ref{thm1.4} is complete.

\subsection{Combined non-relativistic and relaxation limits}\label{sec:3.5}
In the last section, we perform the combined relativistic and
relaxation limits of (\ref{R-E1})-(\ref{R-E3}).\\

\noindent\textit{\underline{Proof of Theorem \ref{thm1.5}.}}
Combined with the scaled variable transform (\ref{R-E7}) where the
superscript $\tau$ is replaced by $(\tau,\varepsilon)$ and the
uniform energy estimate (\ref{R-E11}) in Theorem \ref{thm1.2}, it is
obtained that
$(n^{(\tau,\varepsilon)},\textbf{u}^{(\tau,\varepsilon)},\textbf{E}^{(\tau,\varepsilon)},\textbf{B}^{(\tau,\varepsilon)})$
is a unique solution of the system (\ref{R-E10}) and (\ref{R-E13})
with the superscript $\tau$ replaced by $(\tau,\varepsilon)$,
furthermore, for any fixed $T>0$, we infer that
\begin{eqnarray}n^{(\tau,\varepsilon)}-\bar{n}\in L^\infty_{T}(B^{\sigma}_{2,1})\cap L^{2}_{T}(B^{\sigma}_{2,1}),\label{R-E108}\end{eqnarray}
\begin{eqnarray}
\tau\textbf{u}^{(\tau,\varepsilon)}\in
L^\infty_{T}(B^{\sigma}_{2,1}),\ \ \
\textbf{u}^{(\tau,\varepsilon)}\in
L^2_{T}(B^{\sigma}_{2,1}),\label{R-E109}
\end{eqnarray}
\begin{eqnarray}\textbf{E}^{(\tau,\varepsilon)}\in L^\infty_{T}(B^{\sigma}_{2,1}),\ \ \ \sqrt{\varepsilon}\textbf{E}^{(\tau,\varepsilon)}\in L^{2}_{T}(B^{\sigma}_{2,1}),\label{R-E110}\end{eqnarray}
\begin{eqnarray}\textbf{B}^{(\tau,\varepsilon)}-\overline{\mathbf{B}}\in L^\infty_{T}(B^{\sigma}_{2,1}),\ \ \ \frac{\nabla\textbf{B}^{(\tau,\varepsilon)}}{\sqrt{\tau\varepsilon}}\in L^{2}_{T}(B^{\sigma-1}_{2,1}),\label{R-E111}\end{eqnarray}
uniformly in $(\tau,\varepsilon)$. The relation (\ref{R-E111}) turns
out to yield
\begin{eqnarray}
\{\nabla\textbf{B}^{(\tau,\varepsilon)}\}\rightarrow \textbf{0}\ \ \
\ \mbox{strongly in}\ \ \ L^{2}_{T}(B^{\sigma-1}_{2,1}),\ \ \
\mbox{as}\ \ \tau,\varepsilon\rightarrow0.\label{R-E112}
\end{eqnarray}
Moreover, using the equations (\ref{R-E10}), we get
\begin{eqnarray}n^{(\tau,\varepsilon)}_{t}\in L^{2}_{T}(B^{\sigma-1}_{2,1}),\label{R-E113}\end{eqnarray}
\begin{eqnarray}
\tau^2\textbf{u}^{(\tau,\varepsilon)}_{t}\in
L^{2}_{T}(B^{\sigma-1}_{2,1}),\label{R-E114}
\end{eqnarray}
\begin{eqnarray}
\sqrt{\tau\varepsilon}\textbf{E}^{(\tau,\varepsilon)}_{t}\in
L^{2}_{T}(B^{\sigma-1}_{2,1}),\label{R-E115}
\end{eqnarray}
uniformly in $(\tau,\varepsilon)$.

As previously, it follows from the standard weak convergence methods
and compactness lemma in \cite{S} that there exists some function
$(\mathcal{N},\mathcal{U},\mathcal{E})\in \mathcal{C}([0,\infty),
\bar{n}+B^{\sigma}_{2,1})\times L^2([0,\infty),
B^{\sigma}_{2,1})\times \mathcal{C}([0,\infty), B^{\sigma}_{2,1})$
such that the sequences (up to subsequences) as
$\tau,\varepsilon\rightarrow0$, it holds that
\begin{eqnarray}
\{n^{(\tau,\varepsilon)}\}\rightarrow \mathcal{N}\ \ \ \
\mbox{strongly in}\ \ \
\mathcal{C}([0,T],(B^{\sigma-\delta}_{2,1})_{\mathrm{loc}}),\label{R-E116}
\end{eqnarray}
\begin{eqnarray}
\{\tau^2\textbf{u}^{(\tau,\varepsilon)}\}\rightarrow \textbf{0}\ \ \
\ \mbox{strongly
 in}\ \ \ \mathcal{C}([0,T],(B^{\sigma-\delta}_{2,1})_{\mathrm{loc}}),\label{R-E117}
\end{eqnarray}
\begin{eqnarray}
\{\textbf{u}^{(\tau,\varepsilon)}\}\rightharpoonup \mathcal{U}\ \ \
\ \mbox{weakly
 in}\ \ \ L^2_{T}(B^{\sigma}_{2,1}),\label{R-E118}
\end{eqnarray}
\begin{eqnarray}
\{\sqrt{\tau\varepsilon}\textbf{E}^{(\tau,\varepsilon)}\}\rightarrow
\textbf{0}\ \ \ \ \mbox{strongly
 in}\ \ \ \mathcal{C}([0,T],(B^{\sigma-\delta}_{2,1})_{\mathrm{loc}}),\label{R-E119}
\end{eqnarray}
\begin{eqnarray}
\{\textbf{E}^{(\tau,\varepsilon)}\}\rightharpoonup \mathcal{E}\ \ \
\ \mbox{weakly$^{\star}$
 in}\ \ \ L^\infty_{T}(B^{\sigma}_{2,1}),\label{R-E120}
\end{eqnarray}
\begin{eqnarray}
\{\textbf{B}^{(\tau,\varepsilon)}\}\rightharpoonup
\overline{\mathbf{B}}\ \ \ \ \mbox{weakly$^{\star}$
 in}\ \ \ L^\infty_{T}(B^{\sigma}_{2,1}),\label{R-E121}
\end{eqnarray}
for any $T>0$ and $\delta\in(0,1)$. Thus, in the system
(\ref{R-E10}) and (\ref{R-E13}), the uniform bounded properties
(\ref{R-E108})-(\ref{R-E111}) as well as the convergence properties
(\ref{R-E112}) and (\ref{R-E116})-(\ref{R-E121}) allow us to pass to
the limits $\tau,\varepsilon\rightarrow0$ in the sense of
distributions, which implies that $(\mathcal{N},\mathcal{E})$ is a
global weak solution to the drift-diffusion equations (\ref{R-E9})
satisfying (\ref{R-E15}). This concludes the proof of Theorem
\ref{thm1.5}.

\section*{Acknowledgments}
The research of Jiang Xu is partially supported by the NSFC
(11001127), China Postdoctoral Science Foundation (20110490134) and
NUAA Research Funding (NS2010204).


\begin{thebibliography}{10}
\bibitem{Abidi}
{\sc H. Abidi}, {\em Equation de Navier-Stokes avec densit\'{e} et
viscosit\'{e} variables dans l'espace critique}, Revista
Matem\'{a}tica Iberoamericana, 23(2007), pp.~537--586.

\bibitem{A}
{\sc G.~Al$\grave{1}$}, {\em Global existence of smooth solutions of
the $N$-dimensional Euler-Possion model}, SIAM J. Math. Anal., 35
(2003), pp.~389--422.

\bibitem {C} {\sc J.~Y. Chemin}, {\em Perfect Incompressible Fluids},
{Oxford Lecture Ser. Math. Appl}. 14, Oxford University Press, New
York, 1998 (in English).

\bibitem {C2} {\sc J.~Y. Chemin}, {\em Th\'{e}or\`{e}mes d'unicit\'{e} pour le syst\`{e}me de Navier-Stokes
tridimensionnel}, Journal d'Analyse Math\'{e}matique, 77 (1999),
pp.~25--50.

\bibitem {CG} {\sc J.~F. Coulombel and T.~Goudon}, {\em The strong relaxation
limit of the multidimensional isothermal Euler equations}, Trans.
Amer. Math. Soc., 359 (2007), pp.~637--648.


\bibitem {CJW} {\sc G.~Q.~Chen, J.~W.~Jermore and D.~H.~Wang}, {\em Compressible
Euler-Maxwell equations}, Transp. Theory, Statist. Phys., 29 (2000),
pp.~311-331.

\bibitem {D}
{\sc R. Danchin}, {\em Fourier Analysis Methods for PDE's}, (Lecture
Notes), 2005.


\bibitem{DM} {\sc P. Degond and P. A. Markowich}, {\em A steady-state
potential flow model for semiconductors}, Ann. Mat. Pura Appl., IV
(1993), pp.~87--98.

\bibitem {FX} {\sc D.~Y. Fang and J.~Xu}, {\em Existence and asymptotic behavior of
$C^1$ solutions to the multidimensional compressible Euler equations
with damping}, Nonlinear Anal., 70 (2009), pp. 244--261.

\bibitem {FXZ} {\sc D.~Y. Fang, J.~Xu, and T.~Zhang}, {\em Global
exponential stability of classical solutions to the hydrodynamic
model for semiconductors}, Mathematical Models and Methods in
Applied Sciences, 17 (2007), pp.~1507--1530.

\bibitem {G} {\sc I. Gamba}, {\em Stationary transonic solutions of a
one-dimensional hydrodynamic model for semiconductor}, Comm. Partial
Differential Equations, 17 (1992), pp.~553--577.

\bibitem{Gu} {\sc Y. Guo}, {\em Smooth irrotational flows in the large
to the Euler-Poisson system in $\mathbf{R}^{3+1}$}, Commun. Math.
Phys., 195 (1998),  pp.~249--265.

\bibitem{HMW} {\sc L. Hsiao, P. A. Markowich and S. Wang}, {\em The asymptotic behavior of globally smooth solutions of the
multidimensional isentropic hydrodynamic model for semiconductors},
J. Differential Equations, 192 (2003), pp.~111--133.

\bibitem{HJZ}
{\sc L. Hsiao, S. Jiang and P. Zhang}, {\em Global existence and
exponential stability of smooth solutions to a full hydrodynamic
model to semiconductors}, Monatshefte f$\ddot{u}$r Mathematik, 136
(2002), pp.~269--285.

\bibitem {HZ} {\sc L. Hsiao and K. Zhang}, {\em The relaxation of the
hydrodynamic model for semiconductors to the drift-diffusion
equations}, J. Differential Equations, 165 (2000), pp.~315--354.

\bibitem {J} {\sc J.~W.~Jerome}, {\em The Cauchy problem for compressible
hydrodynamic-Maxwell systems: A local theory for smooth solutions},
Differential Integral Equations, 16 (2003), pp.~1345--1368.

\bibitem{K} {\sc T. Kato}, {\em The Cauchy problem for quasi-linear
symmetric hyperbolic systems}, Arch. Ration. Mech. Anal., 58 (1975),
pp.~181--205.

\bibitem{KY} {\sc S. Kawashima and W.-A Yong}, {\em Dissipative Structure and
entropy for hyperbolic systems of Balance laws}, Arch. Ration. Mech.
Anal., 174 (2004), pp.~345--364.

\bibitem{LMM} {\sc H.~L. Li, P. Markowich and M. Mei},  {\em Asymptotic behavior of subsonic entropy solutions of the isentropic Euler-Poisson
equations}, Quart. Appl. Math., 60 (2002), pp.~773--796.

\bibitem{M} {\sc A. Majda}, {\em Compressible Fluid Flow and
Conservation laws in Several Space Variables}, Springer-Verlag:
Berlin/New York, 1984.

\bibitem{MN}
{\sc P. Marcati and R. Natalini}, {\em Weak solutions to a
hydrodynamic model for semiconductors and relaxation to the
drift-diffusion equations}, Arch. Ration. Mech. Anal., 129 (1995),
pp.~129--145.

\bibitem{MRS}
{\sc P. A. Markowich, C. Ringhofer and C. Schmeiser}, {\em
Semiconductor Equations}, Springer-Verlag: Vienna, 1990.

\bibitem{PW1}
{\sc Y.~J.~Peng and S.~Wang}, {\em Convergence of compressible
Euler-Maxwell equations to compressible Euler-Poisson equations},
Chin. Ann. Math., 28 (2007), pp.~583--602.

\bibitem{PW2} {\sc Y.~J.~Peng and S.~Wang}, {\em Convergence of compressible
Euler-Maxwell equations to compressible Euler equations}, Comm.
Partial Differential Equations, 33 (2008), pp.~349--376.

\bibitem{PW3} {\sc Y.~J.~Peng and S.~Wang}, {\em Rigorous derivation of
incompressible e-MHD equations from compressible Euler-Maxwell
equations}, SIAM J. Math. Anal., 40 (2008), pp.~540--565.

\bibitem{S} {\sc J.~Simon}, {\em Compact sets in the space
$L^{p}(0,T;B)$}, Ann. Math. Pura Appl., 146 (1987), pp.~65--96.

\bibitem{X}
{\sc J. Xu}, {\em Relaxation-time limit in the isothermal
hydrodynamic model for semiconductors}, SIAM J. Math. Anal., 40
(2009), pp.~1979--1991.

\bibitem{XY} {\sc J. Xu and W.-A. Yong}, {\em Relaxation-time limits of non-isentropic hydrodynamic models for semiconductors},
J. Differential Equations, 247, (2009), pp.~1777--1795.

\bibitem{Y} {\sc W.~A. Yong}, {\em Entropy and global existence for hyperbolic
balance laws}, Arch. Ration. Mech. Anal., 172 (2004),
pp.~247--266.


\end{thebibliography}
\end{document}